\documentclass[12pt,twoside]{amsart}
\usepackage{amssymb}
\usepackage{amscd}
\usepackage{color}
\title[Behavior of canonical divisors]
{Behavior of canonical divisors under purely inseparable base changes} 
\author{Hiromu Tanaka} 
\subjclass[2010]{14E30, 14M22.}
\keywords{uniruled, Mori fiber space, cone theorem}
\address{Department of Mathematics, Imperial College, London, 180 Queen's Gate, 
London SW7 2AZ, UK} 
\email{h.tanaka@imperial.ac.uk}
\newcommand{\reg}[0]{{\operatorname{reg}}}
\newcommand{\red}[0]{{\operatorname{red}}}

\newcommand{\Spec}[0]{{\operatorname{Spec}}}

\newcommand{\Supp}[0]{{\operatorname{Supp}}}
\newcommand{\Pic}[0]{{\operatorname{Pic}}}

\newcommand{\Ex}[0]{{\operatorname{Ex}}}
\newtheorem{thm}{Theorem}[section]
\newtheorem{lem}[thm]{Lemma}
\newtheorem{cor}[thm]{Corollary}
\newtheorem{prop}[thm]{Proposition}

\theoremstyle{definition}
\newtheorem{ex}[thm]{Example}

\newtheorem{dfn}[thm]{Definition}

\newtheorem{rem}[thm]{Remark}
\newtheorem*{ack}{Acknowledgments}

\newcommand{\MO}{\mathcal{O}}
\newcommand{\R}{\mathbb{R}}
\newcommand{\Q}{\mathbb{Q}}
\newcommand{\Z}{\mathbb{Z}}

\begin{document}

\maketitle

\begin{abstract}
Let $k$ be an imperfect field. 
Let $X$ be a regular variety over $k$ and 
set $Y$ to be the normalization of $(X\otimes_k k^{1/p^{\infty}})_{{\rm red}}$. 
In this paper, we show that $K_Y+C=f^*K_X$ 
for some effective divisor $C$ on $Y$. 
We obtain the following three applications. 
First, we show that a $K_X$-trivial fiber space with non-normal fibers is uniruled. 
Second, we prove that general fibers of Mori fiber spaces are rationally chain connected. 
Third, we obtain a weakening of the cone theorem for surfaces and threefolds defined over an imperfect field. 
\end{abstract}

\tableofcontents

\setcounter{section}{0}

\section{Introduction}

Let $k$ be a field of characteristic $p>0$. 
Let $X$ be a proper regular variety over $k$. 
The main interest of this paper is to study a relation between 
the canonical divisor $K_X$ and purely inseparable base changes. 
More precisely, we would like to reduce some problems to the ones 
for varieties over algebraically closed fields or perfect fields. 
However, the base change $X\otimes_k k^{p^{-\infty}}$, 
where $k^{p^{-\infty}}:=\bigcup_{e \geq 0} k^{p^{-e}}$, is not normal in general (even not reduced). 
By taking the normalization $Y$ of the reduced structure $(X\otimes_k k^{p^{-\infty}})_{\red}$, 
it is natural to consider the relation between $K_X$ and $K_Y$. 
In this paper, we show the following theorem. 

\begin{thm}[Main theorem, Theorem~\ref{bc-main}]\label{0bc-main}
Let $k$ be a field of characteristic $p>0$. 
Let $X$ be a regular variety over $k$ such that $k$ is algebraically closed in $K(X)$. 
Set $Y$ to be the normalization of $(X\otimes_k k^{p^{-\infty}})_{\red}$ and 
let $Y_{\reg}$ be the regular locus of $Y$. 
Let $f:Y \to X$ be the natural morphism. 
Then, there exists an effective $\Z$-divisor $C$ on $Y$ such that 
$$(\omega_{Y/k^{p^{-\infty}}}\otimes_{\MO_Y} \MO_Y(C))|_{Y_{\reg}} \simeq f^*\omega_{X/k}|_{Y_{\reg}}.$$
We have $C=0$ if and only if $X$ is geometrically normal over $k$. 
\end{thm}

In the following, we consider some applications of Theorem~\ref{0bc-main}. 
We can apply this theorem to $K_X$-trivial fibrations with non-normal fibers as follows.

\begin{thm}[Theorem~\ref{uniruled-main}]\label{0uniruled}
Let $k$ be an algebraically closed field of characteristic $p>0$. 
Let $\pi:X \to S$ be a projective $k$-morphism of normal $k$-varieties 
such that $\pi_*\MO_X=\MO_S$. 
Assume that $K_X$ is $\Q$-Cartier and $\pi$-numerically trivial. 
If the generic fiber of $\pi$ is not geometrically normal, then $X$ is uniruled. 
\end{thm}

\begin{rem}
If $X$ is a smooth surface and $S$ is a curve in Theorem~\ref{0uniruled}, 
then $X$ is a quasi-elliptic fibration (cf. \cite[Section~7]{Badescu}), 
that is, a general fiber is a cuspidal cubic curve. 
\end{rem}

In the minimal model theory, a Mori fiber space is one of the central objects. 
In the following result, we see that its general fibers are rationally chain connected.

\begin{thm}[Theorem~\ref{MFS-RCC}]\label{0MFS-RCC}
Let $k$ be a field. 
Let $\pi:X \to S$ be a projective surjective $k$-morphism of 
normal $k$-varieties such that $\pi_*\MO_X=\MO_S$. 
Assume the following conditions. 
\begin{enumerate}
\item{$X$ is $\Q$-factorial.}
\item{$-K_X$ is $\pi$-ample.}
\item{$\rho(X/S)=1$.}
\end{enumerate}
Then general fibers of $\pi$ are rationally chain connected. 
\end{thm}

\begin{rem}
If $k$ is an algebraically closed field of characteristic zero and $X$ is log canonical, 
then Theorem~\ref{0MFS-RCC} follows from \cite[Corollary~1.5]{HM}. 
\end{rem}

We obtain the following corollary. 

\begin{cor}[Corollary~\ref{Fano-torsion}]
Let $k$ be a field. 
Let $X$ be a projective normal variety over $k$. 
Assume the following conditions. 
\begin{enumerate}
\item{$X$ is $\Q$-factorial.}
\item{$-K_X$ is ample.}
\item{$\rho(X)=1$.}
\end{enumerate}
If $D$ is a numerically trivial Cartier divisor on $X$, 
then there exists $n \in \mathbb Z_{>0}$ such that $\MO_X(nD) \simeq \MO_X$. 
\end{cor}

As other applications, 
we obtain a weakening of cone theorems for surfaces and threefolds over an arbitrary field of positive characteristic. 

\begin{thm}[Theorem~\ref{surface-cone}]\label{0surface-cone}
Let $k$ be a field of characteristic $p>0$. 
Let $X$ be a projective normal surface over $k$ and 
let $\Delta$ be an effective $\R$-divisor such that $K_X+\Delta$ is $\R$-Cartier. 
Let $A$ be an ample $\R$-Cartier $\R$-divisor.
Then, there exist finitely many curves $C_1, \cdots, C_m$ such that 
$$\overline{NE}(X)=\overline{NE}(X)_{K_X+\Delta+A \geq 0}+\sum_{i=1}^m \mathbb R_{\geq 0}[C_i].$$
\end{thm}

\begin{thm}[Theorem~\ref{weak-cone}]\label{0threefold-cone}
Let $k$ be a field of characteristic $p>0$. 
Let $X$ be a projective normal $\Q$-factorial threefold over $k$ and 
let $\Delta$ be an effective $\R$-divisor whose coefficients are at most $1$. 
If $K_X+\Delta$ is not nef, then there exist an ample $\Q$-divisor $A$ and 
finitely many curves $C_1, \cdots, C_m$ such that 
$K_X+\Delta+A$ is not nef and that 
$$\overline{NE}(X)=\overline{NE}(X)_{K_X+\Delta+A \geq 0}+\sum_{i=1}^m \mathbb R_{\geq 0} [C_i].$$
\end{thm}

\medskip
\textbf{Proof of the main theorem:}  
We overview the proof of Theorem~\ref{0bc-main}. 
Let $k$ be a field of characteristic $p>0$. 
Let $X$ be a regular variety over $k$. 
Set $Y$ to be the normalization of $(X\otimes_k k^{p^{-\infty}})_{\red}$. 
Note that $X\otimes_k k^{p^{-\infty}}$ is irreducible 
because $\Spec\,k^{p^{-\infty}} \to \Spec\,k$ is a universally homeomorphism. 
We assume that $X$ is not geometrically reduced, 
otherwise the proof is easy.

By our assumption, $X\otimes_k k^{p^{-\infty}}$ is not reduced. 
However it is difficult to 
compare dualizing sheaves of $X\otimes_k k^{p^{-\infty}}$ and $(X\otimes_k k^{p^{-\infty}})_{\red}$. 
Thus our main idea is to avoid non-reduced schemes. 
For this, we use the following lemma. 

\begin{lem}[Lemma~\ref{p-change-field}]\label{0-change-field}
Let $k$ be a field of characteristic $p>0$. 
Let $F/k$ and $k'/k$ be field extensions. 
Assume the following two conditions. 
\begin{enumerate}
\item{$k$ is purely inseparably closed in $F$, that is, if $x\in F$ satisfies $x^p \in k$, 
then $x\in k$.}
\item{The field extension $k'/k$ is purely inseparable and $[k':k]=p$. }
\end{enumerate}
Then $F\otimes_k k'$ is a field. 
\end{lem}

Set $F:=K(X)$. 
Since $k$ is algebraically closed in $F=K(X)$ by the assumption, 
we see that $k$ is purely inseparably closed in $F$.
Let $k'/k$ be an arbitrary purely inseparable extension with $[k':k]=p$. 
From a pair $(X, k)$, we construct a \lq\lq bigger" pair $(X_1, k_1)$. 
By Lemma~\ref{0-change-field}, we see that $X\otimes_k k'$ is an integral scheme but may not be normal. 
By taking the normalization $X_1$ of $X\otimes_k k'$, $X_1$ is a normal variety over $k'$. 
Let $k_1$ be the purely inseparable closure of $k'$ in $K(X_1)$. 
Since we can ignore codimension two closed subsets, 
we may assume that $X_1$ is regular. 
Then $k_1$ is purely inseparably closed in $K(X_1)$. 
By the inductive argument, 
we obtain sequences $k=:k_0 \subsetneq k_1 \subsetneq k_2 \subsetneq \cdots$ and 
$X=:X_0 \leftarrow X_1 \leftarrow X_2 \leftarrow \cdots$. 

Let us calculate the relation between dualizing sheaves $\omega_{X/k}$ and $\omega_{X_1/k_1}$. 
To obtain $(X_1, k_1)$ from $(X, k)$, 
we apply the following three operations: 
a base change, normalization, replacing a base field $k'$ by $k_1$ (enlarging base field). 
Under these operations, dualizing sheaves can be calculated as follows: 
\begin{itemize}
\item{$\omega_{X\otimes_k k'/k'}\simeq \beta^*\omega_{X/k}$ (base change)}
\item{$\omega_{X_1/k'} \otimes_{\MO_{X_1}} \MO_{X_1}(C_1)\simeq \nu^*\omega_{X\otimes_k k'/k'}$ (normalization)}
\item{$\omega_{X_1/k_{1}}\simeq \omega_{X_{1}/k'}$ (enlarging a base field), }
\end{itemize}
where $\beta:X\otimes_k k' \to X$ is the projection, 
$\nu:X_1 \to X\otimes_k k'$ is the normalization and $C_1$ is an effective divisor. 
Summarizing above, we obtain 
$$K_{X_1}+C_1=g^*K_X$$
where $g:X_1 \to X$ is the induced morphism. 

Therefore, it suffices to find, in advance, a finite purely inseparable extension $k \subset k_{\rm goal}$ 
such that $k=:k_0 \subset k_1 \subset k_2 \subset \cdots \subset k_{\rm goal}$ and 
that the end result $(X_{\rm goal}, k_{\rm goal})$ satisfies a good property 
($X_{\rm goal}$ is geometrically reduced over $k_{\rm goal}$ and so on). 
Such a pair $(X_{\rm goal}, k_{\rm goal})$ can be found as a finite extension model of 
the pair $((X\times_k k^{p^{-\infty}})^N_{\red}, k^{p^{-\infty}})$ where 
$(X\times_k k^{p^{-\infty}})^N_{\red}$ is the normalization of $(X\times_k k^{p^{-\infty}})_{\red}$. 

For more details, see Section~\ref{section-field} and Section~\ref{section-main}. 
In Example~\ref{double-conic}, we give an explicit calculation for the double conic curve.

\medskip
\textbf{Related results:}  
%As a related result of the main theorem (Theorem~\ref{0bc-main}), 
%Chen--Zhang show a similar result in \cite[Theorem~2.4]{CZ}, 
%although the settings, assumptions and proofs are different. 
%\cite{CZ} assumes that all the varieties are projective and Gorenstein. 
%Their proof is a beautiful application of the duality theory. 
%On the other hand, our proof is field-theoretic and constructive. 
%\cite[Theorem~2.4]{CZ} states the existence of an effective divisor $C \geq 0$, 
%however Theorem~\ref{0bc-main}(2) states that it is a non-zero effective divisor $C > 0$ in some cases. 
%
As a related topic, there is a classical result called Tate's genus change (cf. \cite{Schroer1} \cite{Tate}). 
It compares the genus of a given regular curve over an imperfect field 
and the one of the normalization of its purely inseparable base change of degree $p$. 

In Theorem~\ref{0MFS-RCC}, we show that general fibers of a Mori fiber space are rationally chain connected. 
In \cite{KMM}, Koll\'ar--Miyaoka--Mori show that every smooth Fano variety, 
defined over an algebraically closed field of any characteristic, is rationally chain connected. 
After that, Hacon--M\textsuperscript{c}Kernan (\cite{HM}) and \cite{Zhang} prove that 
log Fano varieties in characteristic zero are rationally connected. 
In positive characteristic, there are few results around this problem (cf. \cite{6}). 

In Theorem~\ref{0surface-cone} and Theorem~\ref{0threefold-cone}, 
we establish a cone theorem for surfaces and threefolds over a non-closed field. 
If $k$ is algebraically closed, 
then Theorem~\ref{0surface-cone} and Theorem~\ref{0threefold-cone} 
follows from \cite{T} and \cite{CTX}, respectively. 
If $X$ is a smooth projective geometrically connected variety over a non-closed field, 
then Mori's cone theorem (\cite[Theorem~1.24]{KM}) is established in Mori's original paper (\cite[Ch~2, \S3]{Mori}). 
Note that Mori's cone theorem is not known even for regular projective varieties 
because his bend and break technique depends on the smoothness assumption. 

In \cite[Theorem~2.4]{CZ}, 
Chen--Zhang consider a behavior of canonical bundles under base changes 
to attack the Iitaka conjecture in positive characteristic. 
Their setting differs from our main theorem (Theorem~\ref{0bc-main}). 
They consider a positive dimensional base although we treat the case when the base scheme is a field. 
On the other hand, they assume that the generic fiber is geometrically reduced, 
however we do not impose such an assumption.

\begin{ack}
The author would like to 
thank Professors Paolo Cascini, J\'anos~Koll\'ar, 
Joseph Lipman, Chenyang Xu, Lei Zhang for valuable comments and fruitful discussions. 
The author also thanks to the referee for many constructive suggestions. 
This work is partially supported by JSPS KAKENHI Grant Number 24224001. 
\end{ack}

%%%%%

\section{Preliminaries}

\subsection{Notation}\label{subsection-notation}

We will not distinguish the notations 
invertible sheaves and divisors. 
%For example, we will write $L+M$ for 
%invertible sheaves $L$ and $M$. 
We say $L$ is an $\R$-{\em invertible sheaf} if $L \in \Pic(X)\otimes_{\Z} \R$. 
%For a coherent sheaf $F$ and a Cartier divisor $L$, 
%we define $F(L):=F\otimes \mathcal O_X(L)$. 
A scheme $X$ is {\em normal} if the local ring $\MO_{X, x}$ for every point $x \in X$ 
is an integrally closed integral domain. 
A scheme $X$ is {\em Gorenstein} (resp. {\em Cohen--Macaulay}, resp. $S_2$) 
if so is the local ring $\MO_{X, x}$ for every point $x \in X$. 
In particular, a Gorenstein scheme is Cohen--Macaulay.

We say $X$ is a {\em variety} over a field $k$ (or a $k$-{\em variety}) if 
$X$ is an integral scheme which is separated and of finite type over $k$. 
Note that, in general, this property is not stable under base changes. 
We say $X$ is a {\em curve} (resp. a {\em surface}, resp. a {\em threefold}) 
if $X$ is a variety over $k$ with $\dim X=1$ (resp. $\dim X=2$, resp. $\dim X=3$). 
For a scheme $X$, set $X_{\red}$ to be 
the reduced scheme whose underlying topological space is equal to $X$. 

Let $\Delta$ be an $\R$-divisor on a normal scheme. 
We write $\Delta \leq a$ if, for the prime decomposition $\Delta=\sum_{i\in I}\delta_i \Delta_i$, 
$\delta_i \leq a$ holds for every $i\in I$.

Let $k$ be a field. 
Let $X$ be a separated scheme of finite type over $k$ and 
let $C \hookrightarrow X$ be a closed immersion such that $C$ is a proper curve over $k$. 
Let $L$ be an invertible sheaf on $X$. 
%It is well-known that 
%$$\chi(C, mL|_C)=\dim_k(H^0(C, mL|_C))-\dim_k(H^1(C, mL|_C)) \in\mathbb Q[m]$$ 
%and that its degree is at most one. 
We define the intersection number $L\cdot C$ by 
$$L \cdot C=\chi(C, L|_C)-\chi(C, \MO_C).$$
It is well-known that this number is an integer. 
%Note that this number depends on the base field e.g. 
%$\chi(C, \MO_C) 
Let $f:X \to Y$ be a proper morphism of noetherian schemes, 
and let $L$ and $M$ be an $\R$-invertible sheaves on $X$. 
We say $L$ and $M$ are {\em $f$-numerically equivalent}, written $L \equiv_f M$, if $L \cdot C=M \cdot C$ 
for every proper curve $C$ on $X$ over a closed point $y \in Y$. 
We say $L$ is {\em $f$-numerically trivial}, written $L \equiv_f 0$, 
if $L$ and $\MO_X$ are $f$-numerically equivalent. 
If $Y=\Spec\,k$ for a field $k$, then we merely wright $L \equiv 0$. 
%We say $L$ and $M$ are $f$-numerically equivalent if $L-M \equiv 0$. 

We will freely use the notation and terminology in \cite{KM} and \cite{Kollar2}. 
For the definition of uniruled and rationally chain connected varieties, 
see \cite[Ch IV, Definition~1.1 and Definition~3.2]{Kollar1}. 
For the definition of dualizing sheaves and canonical divisors, 
see Subsection~\ref{subsection-dualizing}.

For a field $k$ of characteristic $p>0$, 
we fix an algebraic closure $\overline k$ and 
set $k^{p^{-\infty}}:=\bigcup_{e \geq 0}\{a \in \overline k\,|\, a^{p^e}\in k\}.$ 

\begin{dfn}
Let $A \subset B$ be a ring extension of $\mathbb{F}_p$-algebras.  
\begin{enumerate}
\item{The {\em purely inseparable closure} $C$ of $A$ in $B$ is the intermediate ring 
$A\subset C \subset B$ such that 
for every $b\in B$ with $b^{p^e}\in A$ for some exponent $e \geq 0$, 
then $b\in C$.}
\item{$A$ is {\em purely inseparably closed} in $B$ if 
the purely inseparable closure $C$ of $A$ in $B$ satisfies $A=C$. }
\end{enumerate}
\end{dfn}

Note that $A$ is purely inseparably closed in $B$ if and only if 
an element $\gamma\in B$ satisfies $\gamma^{p}\in A$, then $\gamma\in A$.

\subsection{Basic properties of purely inseparable base changes}

In this subsection, we summarize basic properties of purely inseparable base changes. 
Some of them may be known results.

\begin{lem}\label{geom-irreducible}
Let $k$ be a field. 
Let $X$ be a normal variety over $k$. 
Then the following assertions hold. 
\begin{enumerate}
\item{If $X$ is geometrically connected, then $X$ is geometrically irreducible over $k$. }
\item{Let $k \subset k'$ be a field extension. 
Then the normalization morphism of $(X\otimes_k k')_{\red}$ is a universal homeomorphism. }
\end{enumerate}
\end{lem}

\begin{proof}
Note that if $X$ is normal and connected, then $X$ is integral. 
Thus if $k$ is characteristic zero, then both of the assertions are clear. 
We assume that $k$ is of characteristic $p>0$. 

Assuming (2), we prove (1). 
We apply (2) for the algebraic closure $k':=\overline k$. 
Then we see that the normalization morphism 
$(X\otimes_k \overline k)_{\red}^N \to (X\otimes_k \overline k)_{\red}$ 
is a universal homeomorphism. 
Since $(X\otimes_k \overline k)_{\red}$ is connected by our assumption, so is $(X\otimes_k \overline k)_{\red}^N$. 
Thus $(X\otimes_k \overline k)_{\red}^N$ is connected and normal, hence it is integral. 
Therefore $X\otimes_k \overline k$ is irreducible because 
the composition morphism 
$$(X\otimes_k \overline k)_{\red}^N \to (X\otimes_k \overline k)_{\red} \to X\otimes_k \overline k$$ 
is a universally homeomorphism. 
This implies (1). 

\medskip

It suffices to show (2). 
We reduce the proof to the case when $k \subset k'$ is a purely inseparable extension. 
For a field extension $k \subset k'$, we obtain 
the following decomposition 
$$k \subset k_1 \subset k_2 \subset k',$$
where $k_1/k$ is purely transcendental, $k_2/k_1$ is algebraic separable, 
and $k'/k_2$ is purely inseparable. 
We can check that $X\times_k k_1$ is normal and hence $X\times_k k_2$ is also normal. 
Therefore, we may assume that $k \subset k'$ is purely inseparable. 

\medskip

%We prove (2) assuming that $k \subset k'$ is purely inseparable. 
%By a standard argument, we may assume that $k \subset k'$ is a finite extension. 
Set $Y$ to be the normalization of $(X\otimes_k k')_{\red}$. 
Then, we obtain the following commutative diagram 
$$\begin{CD}
@. Y \\
@. @VVh V\\
X @<\beta<< (X\otimes_k k')_{\red}\\
@VVV @VVV\\
\Spec\,k @<<< \Spec\,k'.
\end{CD}$$
Since $X$ and $Y$ are normal and $Y \to X$ is an affine integral surjective morphism 
(i.e. the corresponding ring extensions are integral), 
$Y$ can be obtained by the integral closure of $X$ in $K(Y)$. 
Since the field extension $K(Y)/K(X)$ is purely inseparable, 
the ring extensions 
$$\MO_X(U) \hookrightarrow \MO_{(X\otimes_k k')_{\red}}(\beta^{-1}(U)) \hookrightarrow 
\MO_Y(h^{-1}\beta^{-1}(U))$$
is purely inseparable for every affine open subset $U$ on $X$. 
Since $\beta:(X\otimes_k k')_{\red} \to X$ is a universally homeomorphism, 
the ring extension 
$$\MO_{(X\otimes_k k')_{\red}}(V) \hookrightarrow \MO_Y(h^{-1}V)$$
is purely inseparable for every affine open subset $V$ of $(X\otimes_k k')_{\red}$. 
Therefore, $h$ is a universally homeomorphism. 
%Since $Y \to X$ is finite, an iterated Frobenius $F^e_X$ factors through this: 
%$$F^e_X:X \to  
\end{proof}

Positivity of intersection number does not change under base changes. 

\begin{lem}\label{intersection-bc}
Let $k \subset k'$ be a field extension. 
Let $C$ be a proper $k$-curve and 
fix a proper $k'$-curve $B$ equipped with a closed immersion $B \hookrightarrow C\otimes_k k'$. 
Set $\beta:B \hookrightarrow C\otimes_k k' \to C$ to be the composite morphism. 
Let $L$ be an invertible sheaf on $C$. 
Then, the following assertions hold. 
\begin{enumerate}
\item{$L\cdot C > 0$ if and only if $\beta^*L\cdot B> 0$. }
\item{$L\cdot C < 0$ if and only if $\beta^*L\cdot B< 0$. }
\item{$L\cdot C = 0$ if and only if $\beta^*L\cdot B= 0$. }
\end{enumerate}
\end{lem}

\begin{proof}
(2) follows from (1). 
(3) holds by (1) and (2). 
Thus, we only show (1). 

By a standard argument, 
we can assume that $k \subset k'$ is a finite extension. 
Then, $B \to C$ is a finite surjective $k$-morphism between proper $k$-curves. 
The assertion holds from the fact that 
$L$ is ample if and only if $\beta^*L$ is ample.
\end{proof}

By a purely inseparable base change, 
Picard numbers do not change.

\begin{prop}\label{p-insep-picard}
Let $k$ be a field of characteristic $p>0$. 
\begin{enumerate}
\item{
Let $f:X \to Y$ be a finite $k$-morphism between proper $k$-schemes. 
If $f$ is a universal homeomorphism, then $\rho(X)=\rho(Y)$.}
\item{Let $k \subset k'$ be a (possibly infinite) purely inseparable field extension. 
Let $X$ be a proper scheme over $k$. 
Then, $\rho(X)=\rho(X\otimes_k k')$.}
\item{Let $X$ be a proper normal variety over $k$. 
Set $Y$ to be the normalization of $(X\otimes_k k^{p^{-\infty}})_{\red}$. 
Then, $\rho(X)=\rho(Y)$. }
\end{enumerate}
\end{prop}

\begin{proof}
(1) 
See \cite[Lemma~1.4(3)]{Keel}.

(2) 
Let $\beta:X\otimes_k k' \to X$ be the projection. 
Fix an invertible sheaf $L$ on $X$. 
By Lemma~\ref{intersection-bc}, 
$L \equiv 0$ if and only if $\beta^*L\equiv 0$. 
Thus, we obtain a $\Q$-linear map 
$$\tilde \beta:({\rm Pic}(X)/\equiv)_{\Q} \to ({\rm Pic}(X\otimes_k k')/\equiv)_{\Q},\,\,\, L \mapsto \beta^*L$$ 
which is injective, where $M_{\Q}:=M\otimes_{\Z} \Q$ for a $\Z$-module $M$. 
We show that $\tilde \beta$ is surjective. 
Let $L$ be an invertible sheaf on $X\otimes_k k'$. 
Then, we can find an intermediate field $k\subset k_1 \subset k'$ 
such that $k \subset k_1$ is a finite extension and that 
$L$ is defined over $k_1$, that is, there is an invertible sheaf on $L_1$ on $X\otimes_k k_1$ 
whose pull-back is $L$. 
By (1), we can find $L_X\in {\rm Pic}(X)_{\Q}$ such that $\beta_1^*L_X \equiv L$ 
where $\beta_1:X \otimes_k k_1 \to X$. 
This implies the surjectivity of $\tilde \beta$. 

(3) 
By Lemma~\ref{geom-irreducible}, 
$Y \to X\otimes_k k^{p^{-\infty}}$ is a universal homeomorphism. 
Then, the assertion follows from (1) and (2). 
\end{proof}

A purely inseparable cover of a $\Q$-factorial variety is $\Q$-factorial. 

\begin{lem}\label{p-insep-qfac}
Let $f:X \to Y$ be a morphism 
of noetherian integral normal schemes. %which are separated over $\Spec\,\mathbb Z$. 
If $f$ is homeomorphic and $Y$ is $\Q$-factorial, then $X$ is $\Q$-factorial. 
\end{lem}

\begin{proof}
Let $D_X$ be a prime divisor on $X$. 
Note that prime divisors are irreducible closed subsets of codimension one. 
Thus $\Supp(f(D_X))=:D_Y$ is also a prime divisor on $Y$. 
Since $Y$ is $\Q$-factorial, $mD_Y$ is Cartier for some $m\in\mathbb Z_{>0}$. 
Then, $f^*(mD_Y)$ is also Cartier and $\Supp(f^*(mD_Y))=\Supp D_X$. 
Thus, $f^*(mD_Y)=nD_X$ for some $n\in\mathbb Z_{>0}$. 
\end{proof}

%%%%%%%
\subsection{Known results on dualizing sheaf}\label{subsection-dualizing}

We recall the definition of dualizing sheaves and 
collect some basic properties. 
For more details, see \cite{Conrad}, \cite{RD}, and \cite{LH}. 
%Note that we only treat separated morphisms of finite type, 
%hence automatically they are embeddable in the sense of 

For a separated and of finite type morphism of noetherian schemes $f:X \to Y$, 
there is a functor 
$$f^!:D^+_{{\rm qc}}(Y) \to D^+_{{\rm qc}}(X)$$
which satisfies the following properties (cf. \cite[Ch. III, Theorem~8.7]{RD}) 
\begin{itemize}
\item{$(f \circ g)^! \simeq g^!f^!$ for two such morphisms $f:X \to Y$ and $Y \to Z$. }
\item{If $f$ is a smooth morphism of relative pure dimension $n$, 
then $f^!G^{\bullet} \simeq f^*(G^{\bullet}) \otimes_{\MO_X} \left(\bigwedge^n\Omega_{X/Y}\right)[n].$}
\item{If $f$ is a finite morphism, then 
$f^!G^{\bullet} \simeq \overline{f}^* RHom_{\MO_Y}(f_*\MO_X, G^{\bullet})$, 
where $\overline{f}^*$ is the natural exact functor (cf. \cite[Ch. III, \S 6]{RD}) 
$$\overline{f}^*:{\rm Mod}(f_*\MO_X) \to {\rm Mod}(X).$$ }
\end{itemize}
Note that, in general, we need the Nagata compactification to define $f^!$.

\begin{dfn}
Let $k$ be a field. 
Let $X$ be a $d$-dimensional separated scheme of finite type over $k$. 
We set 
$$\omega_{X/k}:=\mathcal{H}^{-d}(\alpha^!\MO_{\Spec\,k}),$$ 
where $\alpha:X \to \Spec\,k$ is the structure morphism. 
\end{dfn}

A dualizing sheaf does not change by enlarging a base field. 

\begin{lem}\label{base-invariance}
Let $k/k_0$ be a finite field extension. 
Let $X$ be a separated scheme of finite type over $k$. 
Note that $X$ is also of finite type over $k_0$. 
Then, there exists an isomorphism
$$\omega_{X/k}\simeq \omega_{X/k_0}.$$ 
\end{lem}

\begin{proof}
Set $\alpha:X \overset{\beta}\to \Spec\,k \overset{\theta}\to \Spec\,k_0.$ 
By $\alpha^!=\beta^!\theta^!$, it suffices to show 
$\theta^!\MO_{\Spec\,k_0} \simeq \MO_{\Spec\,k}$. 
This follows from ${\rm Hom}_{k_0}(k, k_0) \simeq k$. 
\end{proof}

Thanks to Lemma~\ref{base-invariance}, 
we can define canonical divisors $K_X$ independent of a base field.

\begin{dfn}
Let $k$ be a field. 
If $X$ is a normal variety over $k$, 
then it is well-known that $\omega_{X/k}$ is a reflexive sheaf. 
Let $K_X$ be a divisor which satisfies 
$$\MO_X(K_X)\simeq \omega_{X/k}.$$
Such a divisor $K_X$ is called {\em canonical divisor}. 
Note that a canonical divisor is determined up to linear equivalence. 
\end{dfn}

We need the following result.

\begin{thm}\label{K-bc}
Let $k'/k$ be a field extension. 
Let $X$ be a pure dimensional Cohen--Macaulay separated scheme of finite type over $k$. 
Let $\beta:X\otimes_k k' \to X$ be the projection. 
Then, 
$$\beta^*\omega_{X/k} \simeq \omega_{X\otimes_k k'/k'}.$$
\end{thm}

\begin{proof}
See, for example, 
\cite[Theorem 3.6.1]{Conrad} or \cite[Theorem~4.4.3]{LH}. 
\end{proof}

Before stating a result on normalizations, let us recall some notation and basic properties. 
Let $X$ be a variety over a field $k$ and let $\nu:X^N \to X$ be the normalization. 
Let $\mathcal D$ be the closed subscheme on $X$ 
defined by the coherent ideal sheaf $\mathcal Hom_{\MO_X}(\nu_*\MO_{X^N}, \MO_X) \subset \MO_X.$ 
Set $\mathcal C:=\mathcal D \times_X X^N$. 
We call $\mathcal C$ the {\em conductor} scheme of $\nu$. 
If $X$ is $S_2$, then $\mathcal C$ is of pure codimension one. 
Under the normalization, dualizing sheaves are changed as follows. 

\begin{prop}\label{K-normalization}
Let $k$ be a field. 
Let $X$ be a Gorenstein variety over $k$. 
Set $\nu:X^N \to X$ to be the normalization and 
let $(X^N)_{\rm reg}$ be the regular locus of $X^N$. 
Then, 
$$(\omega_{X/k}\otimes_{\MO_X} \MO_X(C))|_{(X^N)_{\rm reg}}\simeq (\nu^*\omega_{X/k})|_{(X^N)_{\rm reg}}$$
where $C$ is an effective $\Z$-divisor on $X^N$ 
whose support is the same as the support 
%codimension one part 
of the conductor scheme $\mathcal C$ of $\nu$. 
\end{prop}

\begin{proof}
See \cite[Proposition~2.3]{Reid}. 
\end{proof}

%%%%%
\section{Field theoretic version of the main theorem}\label{section-field}

The main result of this section is Proposition~\ref{field-main}, 
which plays a crucial role in Section~\ref{section-main} (cf. Lemma~\ref{variety-main}). 
For this, we need the following three lemmas.

\begin{lem}\label{geom-red-perfect}
Let $k$ be a field of characteristic $p>0$. 
Let $F$ be a field containing $k$. 
If $F$ is geometrically reduced over $k$, then $k$ is purely inseparably closed in $F$. 
\end{lem}

\begin{proof}
The assertion follows from \cite[Ch. V, \S 15, 4. Mac Lane's separability criterion, Theorem 2(e)]{Bourbaki}. 
%Assume that there exist $\alpha\in k$ such that 
%$\alpha^{1/p}\in F\setminus k$. 
%Let us derive a contradiction. 
%Set $k':=k[\alpha^{1/p}] \subset F$. 
%This induces an injection 
%$$k'\otimes_{k} k' \hookrightarrow F\otimes_{k} k'.$$
%Thus, it suffices to prove that 
%$k'\otimes_{k} k'$ is not reduced. 
%This follows from 
%$$k'\otimes_{k} k'\simeq k'[t]/(t^p-\alpha)=k'[t]/(t-\alpha^{1/p})^p.$$
\end{proof}

\begin{lem}\label{fg-noetherian}
Let $k$ be a field and let $F$ be a finitely generated field over $k$. 
Let $A$ be an arbitrary noetherian $k$-algebra. 
Then, $F\otimes_k A$ is a noetherian ring. 
\end{lem}

\begin{proof}
We can find a finitely generated $k$-subalgebra $R\subset F$ whose fractional field is $F$. 
Set $S:=R\setminus \{0\}$ and we see $F=S^{-1}R$. 
We can check that there exists the following ring isomorphism 
$$F\otimes_k A = (S^{-1}R) \otimes_k A \simeq T^{-1}(R\otimes_k A)$$
where $T:=\{s\otimes 1\,|\, s\in S\} \subset R\otimes_k A$. 
By the Hilbert basis theorem, $R\otimes_k A$ is a noetherian ring. 
Therefore, so is $F\otimes_k A\simeq T^{-1}(R\otimes_k A)$. 
\end{proof}

The following lemma is a little bit generalization of \cite[Lemma~1.3]{Schroer}.

\begin{lem}\label{p-change-field}
Let $k$ be a field of characteristic $p>0$. 
Let $F/k$ and $k'/k$ be field extensions. 
Assume the following two conditions. 
\begin{enumerate}
\item{$k$ is purely inseparably closed in $F$.}
\item{The field extension $k'/k$ is purely inseparable and $[k':k]=p$. }
\end{enumerate}
Then, $F\otimes_k k'$ is a field. 
\end{lem}

\begin{proof}
Fix $\alpha\in k'\setminus k$ and we see $\beta:=\alpha^p\in k$. 
By (2), we see $k'=k[\alpha]\simeq k[t]/(t^p-\beta).$ 
This implies 
$$F\otimes_k k'\simeq F[t]/(t^p-\beta).$$
It is enough to show that there is no element $\gamma\in F$ such that $\gamma^p=\beta$. 
If not, the field extensions $k\subset k[\gamma]\subset F$ violate the condition (1). 
\end{proof}

We show the main result in this section. 

\begin{prop}\label{field-main}
Let $k$ be a field of characteristic $p>0$. 
Let $F$ be a finitely generated field over $k$. 
Then, there exist sequences of field extensions 
\begin{itemize}
\item{$k=:k_1^0\subset k_1 \subset k_2^0 \subset k_2 \subset \cdots \subset k_n^0 \subset k_n \subset k^{p^{-\infty}}$, and}
\item{$F=:F_1 \subset F_2\subset \cdots \subset F_n \subset (F\otimes_k k^{p^{-\infty}})_{\red}$, }
\end{itemize}
which satisfy the following properties. 
\begin{enumerate}
\item{The field extension $k_n/k$ is a finite purely inseparable extension.}
\item{The field extension $F_n/F$ is a finite purely inseparable extension.}
\item{For every $i$, we obtain $k_i^0\subset k_i \subset F_i$ and $k_i$ is the purely inseparable closure of $k_i^0$ in $F_i$. }
\item{For every $i$, $[k_{i+1}^0:k_i]=p$. }
\item{For every $i$, $F_{i+1}\simeq F_i\otimes_{k_i}k_{i+1}^0$.}
\item{$F_n$ is geometrically reduced over $k_n$. }
\item{The canonical map 
$\theta:F_n \otimes_{k_n} k^{p^{-\infty}} \to (F\otimes_k k^{p^{-\infty}})_{\red}$ 
is bijective. }
\end{enumerate}
\end{prop}

\begin{proof}
For a ring $R$, set ${\rm Nil}_R$ to be its nilradical. 
Consider the ring $F\otimes_{k} k^{p^{-\infty}}$. 
This ring is a noetherian ring by Lemma~\ref{fg-noetherian}. 
Thus, its nilradical ${\rm Nil}_{F\otimes_{k} k^{p^{-\infty}}}$ is finitely generated. 
Then we can find a finite purely inseparable extension 
$k_{\rm goal}/k$ with $k_{\rm goal} \subset k^{p^{-\infty}}$ 
such that 
$${\rm Nil}_{F\otimes_k k_{\rm goal}}\cdot(F\otimes_{k} k^{p^{-\infty}})={\rm Nil}_{F\otimes_{k} k^{p^{-\infty}}}.$$
Therefore, we obtain 
\begin{eqnarray*}
(F\otimes_k k_{\rm goal})_{\red}\otimes_{k_{\rm goal}} k^{p^{-\infty}} 
&=&((F\otimes_k k_{\rm goal})/{\rm Nil}_{F\otimes_k k_{\rm goal}})\otimes_{k_{\rm goal}} k^{p^{-\infty}}\\
&=& (F\otimes_k k^{p^{-\infty}})/\left({\rm Nil}_{F\otimes_k k_{\rm goal}}\cdot(F\otimes_{k} k^{p^{-\infty}})\right)\\
&=&(F\otimes_k k^{p^{-\infty}})/{\rm Nil}_{F\otimes_{k} k^{p^{-\infty}}}\\
&=&(F\otimes_k k^{p^{-\infty}})_{\red}.
\end{eqnarray*}
In particular, 
$$F_{\rm goal}:=(F \otimes_{k} k_{\rm goal})_{\red}$$ 
is geometrically reduced over $k_{\rm goal}$. 
Note that $(F\otimes_k k^{p^{-\infty}})_{\red}$ and $F_{\rm goal}$ are fields. 
We see 
$$k_{\rm goal} \subset k^{p^{-\infty}},\,\,\, F_{\rm goal} \subset (F\otimes_{k} k^{p^{-\infty}})_{\red}$$
and 
$$F_{\rm goal}\otimes_{k_{\rm goal}} k^{p^{-\infty}} \simeq (F\otimes_{k} k^{p^{-\infty}})_{\red}.$$
We see that $k_{\rm goal}$ is purely inseparably closed in $F_{\rm goal}$ by Lemma~\ref{geom-red-perfect}. 
In the following argument, we inductively construct two increasing sequences of subfields 
of $k_{\rm goal}$ and of $F_{\rm goal}$, respectively.  
%$$k_1^0\subset k_1 \subset k_2^0 \subset k_2\subset k_3^0 \subset k_3 \subset \cdots \subset k_{\rm goal}.$$
%$$F_1 \subset F_2 \subset F_3 \subset \cdots \subset F_{\rm goal}.$$

Fix an integer $j\geq 1$ and we assume that we have already constructed  two sequences of fields 
$$k_1^0\subset k_1 \subset k_2^0 \subset k_2\subset \cdots \subset k_j^0\subset k_j\subset k_{\rm goal}.$$
$$F_1 \subset F_2 \subset \cdots \subset F_j \subset F_{\rm goal}$$
such that, for each $1\leq i\leq j$, the following properties $(1)_i$--$(5)_i$ hold.  
\begin{enumerate}
\item[$(1)_i$]{The field extension $k_i/k$ is a finite purely inseparable extension.}
\item[$(2)_i$]{The field extension $F_i/F$ is a finite purely inseparable extension.}
\item[$(3)_i$]{$k_i^0\subset k_i \subset F_i$ and $k_i$ is the purely inseparable closure of $k_i^0$ in $F_i$. }
\item[$(4)_i$]{$[k_{i}^0:k_{i-1}]=p$. }
\item[$(5)_i$]{$F_{i} \simeq F_{i-1}\otimes_{k_{i-1}}k_{i}^0$.}
\end{enumerate}
Assuming $k_j \neq k_{\rm goal}$, 
we construct $k_{j+1}^0, k_{j+1}$ and $F_{j+1}$ 
which satisfy 
$$k_j\subset k_{j+1}^0\subset  k_{j+1} \subset k_{\rm goal}$$
$$F_j \subset F_{j+1} \subset F_{\rm goal}$$
and the properties $(1)_{j+1}$--$(5)_{j+1}$. 
Since $k_j \neq k_{\rm goal}$, there is an intermediate field $k_j\subsetneq k_{j+1}^0 \subset k_{\rm goal}$ 
such that $[k_{j+1}^0:k_j]=p$. 
This implies $(4)_{j+1}$. 
Set $F_{j+1}:=F_j\otimes_{k_j}k_{j+1}^0$ and $(5)_{j+1}$ holds. 
Thanks to $(3)_j$ and $(4)_{j+1}$, 
we can apply Lemma~\ref{p-change-field} to $F_{j+1}\simeq F_{j}\otimes_{k_{j}}k_{j+1}^0$, 
and see that $F_{j+1}$ is a field. 
Let $k_{j+1}$ be the purely inseparable closure of $k_{j+1}^0$ in $F_{j+1}$. 
Then $(3)_{j+1}$ holds. 
We show $k_{j+1} \subset k_{\rm goal}$ and $F_{j+1} \subset F_{\rm goal}$. 
Note that these imply $(1)_{j+1}$ and $(2)_{j+1}$, respectively. 
Since $F_j$, $k_j$ and $k_{j+1}^0$ are contained in $F_{\rm goal}$, 
we obtain a natural ring homomorphism 
$$F_{j+1}=F_j\otimes_{k_j}k_{j+1}^0 \to F_{\rm goal}.$$
Since $F_{j+1}$ and $F_{\rm goal}$ are fields, this ring homomorphism is automatically injective. 
Thus, we obtain $F_{j+1} \subset F_{\rm goal}$ by replacing $F_{j+1}$ with the image of 
the injection $F_{j+1} \to F_{\rm goal}$. 
Then we obtain $k_{j+1} \subset k_{\rm goal}$ because 
$k_{j+1}^0 \subset k_{\rm goal}$ and $k_{\rm goal}$ is purely inseparably closed in $F_{\rm goal}$. 

\medskip
By the above inductive construction, 
we obtain two sequences 
$$k_1^0\subset k_1 \subset k_2^0 \subset k_2\subset \cdots \subset k_n^0\subset k_n=k_{\rm goal}$$
$$F_1 \subset F_2 \subset \cdots \subset F_n \subset F_{\rm goal}$$
which satisfy (1)--(5) in the proposition and 
the following two conditions: 
\begin{enumerate}
\item[$(6)'$]{$F_{\rm goal}$ is geometrically reduced over $k_n=k_{\rm goal}$. }
\item[$(7)'$]{$F_{\rm goal} \otimes_{k_{\rm goal}} k^{1/p^{\infty}} \simeq (F\otimes_k k^{p^{-\infty}})_{\red}$.}
\end{enumerate}
Thus, to show (6) and (7), we prove $F_n=F_{\rm goal}$. 
For this, it is enough to prove $F_n\supset F_{\rm goal}$. 
By $F_{\rm goal}=(F_1\otimes_{k_1} k_{\rm goal})_{\red}$, 
it suffices to show that $F_n \supset F_1$ and $F_n \supset k_{\rm goal}$, 
which is obvious. 
\end{proof}

%%%%%
\section{Main theorem: canonical divisors and purely inseparable base changes}\label{section-main}

In this section, we show the main theorem of this paper (Theorem~\ref{bc-main}). 
We start with a lemma which almost follows from Proposition~\ref{field-main}.

\begin{lem}\label{variety-main}
Let $k$ be a field of characteristic $p>0$. 
Let $X$ be a normal variety over $k$. 
Then, there exist sequences 
\begin{itemize}
\item{$k=:k_1^0\subset k_1 \subset k_2^0 \subset k_2 \subset \cdots \subset k_n^0 \subset k_n \subset k^{p^{-\infty}}$,}
\item{$X=:X_1 \leftarrow X_2\leftarrow \cdots \leftarrow X_n \leftarrow ((X\otimes_k k^{p^{-\infty}})_{\red})^N$, and }
\item{$X_i \to \Spec\,k_i\to \Spec\,k_i^0$,}
\end{itemize}
where $((X\times_k k^{p^{-\infty}})_{\red})^N$ is the normalization of $(X\otimes_k k^{1/p^{\infty}})_{\red}$ 
which satisfy the following properties. 
\begin{enumerate}
\item{The field extension $k_n/k$ is a finite purely inseparable extension.}
\item{For every $i$, $X_i$ is a normal variety over $k_i$, and 
the morphism $X_{i} \to X_{i-1}$ is a finite surjective purely inseparable morphism.}
\item{For every $i$, $k_i$ is the purely inseparable closure of $k_i^0$ in $K(X_i)$. }
\item{For every $i$, $[k_{i+1}^0:k_i]=p$. }
\item{For every $i$, 
$X_{i-1}\otimes_{k_{i-1}} k_{i}^0$ is integral and $X_{i}$ is the normalization of $X_{i-1}\otimes_{k_{i-1}} k_{i}^0$.}
\item{$X_n$ is geometrically reduced over $k_n$. }
\item{The induced morphism $((X\otimes_k k^{p^{-\infty}})_{\red})^N \to X_n\otimes_{k_n} k^{p^{-\infty}}$ 
is the normalization of $X_n\otimes_{k_n} k^{p^{-\infty}}$. }
\end{enumerate}
\end{lem}

\begin{proof}
Set $F:=K(X)$. 
Then, we can apply Proposition~\ref{field-main} and 
we obtain sequences of field extensions 
\begin{itemize}
\item{$k=:k_1^0\subset k_1 \subset k_2^0 \subset k_2 \subset \cdots \subset k_n^0 \subset k_n \subset k^{p^{-\infty}}$, and }
\item{$F=:F_1 \subset F_2\subset \cdots \subset F_n \subset (F\otimes_k k^{p^{-\infty}})_{\red}$,}
\end{itemize}
which satisfy the properties (1)--(7) in Proposition~\ref{field-main}. 
%\begin{enumerate}
%\item[$(1)'$]{The field extension $k_n/k$ is a finite purely inseparable extension.}
%\item[$(2)'$]{The field extension $F_n/F$ is a finite purely inseparable extension.}
%\item[$(3)'$]{For every $i$, we obtain $k_i^0\subset k_i \subset F_i$ and $k_i$ is the purely inseparable closure of $k_i^0$ in $F_i$. }
%\item[$(4)'$]{For every $i$, $[k_{i+1}^0:k_i]=p$ for each $i$. }
%\item[$(5)'$]{For every $i$, $F_{i+1}=F_i\otimes_{k_i}k_{i+1}^0$.}
%\item[$(6)'$]{$F_n$ is geometrically reduced over $k_n$. }
%\item[$(7)'$]{The universal property of tensor products induces a ring homomorphism 
%$\theta:F_n \otimes_{k_n} k^{1/p^{\infty}} \to (F\otimes_k k^{p^{-\infty}})_{\red}$. 
%Then $\theta$ is bijecitve. }
%\end{enumerate}
We construct a sequence 
$$X=:X_1 \leftarrow X_2\leftarrow \cdots \leftarrow X_n$$
which has a morphism $X_i \to \Spec\,k_i$ for every $i$ and satisfies properties (1)--(7). 
Since (1) and (4) in Proposition~\ref{field-main} 
imply (1) and (4) in the lemma respectively, it suffices to show (2), (3), (5), (6), and (7) in the lemma. 
For this, we inductively construct a normal $k_i$-variety $X_i$ with $K(X_i)=F_i$.

Set $X=:X_1$. 
Since $X_1$ is normal, the structure morphism 
$X_1 \to \Spec\,k_1^0$ factors through $\Spec\,k_1$. 
Therefore, we can take the fiber product $X_1\otimes_{k_1} k_2^0$. 
The ring corresponding to an open affine subscheme of $X_1\otimes_{k_1} k_2^0$ is contained 
in a field $F_2=F_1\otimes_{k_1} k_2^0$. 
Thus, $X_1\otimes_{k_1} k_2^0$ is an integral scheme. 
Take the normalization $X_2$ of $X_1\otimes_{k_1} k_2^0$ and we see $K(X_2)=F_2$. 
We obtain $X_2 \to \Spec\,k_2\to \Spec\,k_2^0$ and, for $i=2$, 
properties (2), (3), (5) in the lemma follow from (2), (3), (5) in Proposition~\ref{field-main}. 
By the same argument, we can construct $X_1, X_2, X_3, \cdots, X_n$ 
which satisfy (1)--(5) and $K(X_i)=F_i$. 
Moreover, (6) and (7) in the lemma follow from $(6)$ and $(7)$ in Proposition~\ref{field-main}, respectively. 
We are done. 
\end{proof}

We show the main theorem of this paper. 

\begin{thm}\label{bc-main}
Let $k$ be a field of characteristic $p>0$. 
Let $X$ be a normal variety over $k$. 
Set $Y$ to be the normalization of $(X\otimes_{k} k^{p^{-\infty}})_{\red}$ 
and let $f:Y \to X$ be the induced affine morphism. 
Let $X_{\reg}$ and $Y_{\reg}$ be the regular loci of $X$ and $Y$, respectively. 
%Assume that $K_X$ is $\Q$-Cartier. 
Then, there exists an effective $\Z$-divisor $C$ on $Y$ such that 
$$(\omega_{Y/k^{p^{-\infty}}}\otimes_{\MO_Y}\MO_Y(C))|_{Y_{\reg}} \simeq 
f^*(\omega_{X/k}|_{X_{\reg}})|_{Y_{\reg}}.$$
Moreover, $C$ can be chosen to be nonzero if one of the following conditions holds. 
\begin{enumerate}
\item[(a)]{$X$ is geometrically reduced but not geometrically normal over $k$.}
\item[(b)]{$X$ is not geometrically reduced and $k$ is algebraically closed in $K(X)$.}
\end{enumerate}
%the following assertions hold. 
%\begin{enumerate}
%\item[(a)]{If $X$ is geometrically reduced but not geometrically normal over $k$, then 
%there exists a nonzero effective $\Z$-divisor $C$ on $Y$ such that 
%$$(\omega_{Y/k^{p^{-\infty}}}\otimes_{\MO_Y}\MO_Y(C))|_{Y_{\reg}} \simeq f^*(\omega_{X/k}|_{X_{\reg}})|_{Y_{\reg}}.$$}
%\item[(b)]{There exists an effective $\Z$-divisor $C$ on $Y$ such that 
%$$(\omega_{Y/k^{p^{-\infty}}}\otimes_{\MO_Y}\MO_Y(C))|_{Y_{\reg}} \simeq f^*(\omega_{X/k}|_{X_{\reg}})|_{Y_{\reg}}.$$}
%\item[(c)]{If $X$ is not geometrically reduced and $k$ is algebraically closed in $K(X)$, then 
%there exists a nonzero effective $\Z$-divisor $C$ on $Y$ such that 
%$$(\omega_{Y/k^{p^{-\infty}}}\otimes_{\MO_Y}\MO_Y(C))|_{Y_{\reg}} \simeq f^*(\omega_{X/k}|_{X_{\reg}})|%_{Y_{\reg}}.$$}
%\end{enumerate}
\end{thm}

\begin{proof}
We divide the proof into three steps. 

\medskip
\textbf{Step 1:} 
The assertion in the theorem holds under the condition (a). 

\begin{proof}[Proof of Step~1]
%Assume that $X$ is geometrically reduced but not geometrically normal over $k$, 
%and we show the assertion. 
Set 
$$f:Y\overset{\nu}\to X\otimes_{k} k^{p^{-\infty}}\overset{\beta}\to X,$$
where $\nu$ is the normalization. 
Since we can omit codimension two closed subsets, 
we may assume that $X$ and $Y$ are regular. 
Then $X\otimes_{k} k^{p^{-\infty}}$ is Gorenstein and we obtain 
$$\omega_{X\otimes_{k} k^{p^{-\infty}}/k^{p^{-\infty}}} \simeq \beta^*\omega_{X/k}$$
by Theorem~\ref{K-bc}. 
Moreover, by Proposition~\ref{K-normalization}, we obtain 
$$\omega_{Y/k^{p^{-\infty}}}\otimes_{\MO_Y} \MO_Y(C)\simeq \nu^*\omega_{X\times_{k} k^{p^{-\infty}}/k^{p^{-\infty}}}$$
where $C$ is an effective $\Z$-divisor 
whose support is equal to the conductor of $\nu$. 
%Therefore we obtain 
%$$K_Y+C=f^*K_X.$$

We show $C\neq 0$. 
It suffices to show that $X\otimes_{k} k^{p^{-\infty}}$ is not regular in codimension one. 
Assume that $X\otimes_{k} k^{p^{-\infty}}$ is regular and let us derive a contradiction. 
Since $X$ is $S_2$, so is $X\otimes_{k} k^{p^{-\infty}}$. 
Then, by Serre's criterion for normality, $X\otimes_{k} k^{p^{-\infty}}$ is normal. 
This is impossible because $X$ is not geometrically normal over $k$. 
\end{proof}

\medskip
\textbf{Step 2:} 
There exists an effective $\Z$-divisor $C$ on $Y$ such that 
$$(\omega_{Y/k^{p^{-\infty}}}\otimes_{\MO_Y}\MO_Y(C))|_{Y_{\reg}} \simeq 
f^*(\omega_{X/k}|_{X_{\reg}})|_{Y_{\reg}}.$$

\begin{proof}[Proof of Step~2]
%By Proposition~\ref{purely-enlarging}, 
%we can replace $k$ with its purely inseparable closure in $K(X)$. 
%Thus we can assume that $k$ is purely inseparable closed in $K(X)$. 
We apply Proposition~\ref{variety-main} and obtain sequences 
\begin{itemize}
\item{$k=:k_1^0\subset k_1 \subset k_2^0 \subset k_2 \subset \cdots \subset k_n^0 \subset k_n \subset k^{p^{-\infty}}$,}
\item{$X=:X_1 \leftarrow X_2\leftarrow \cdots \leftarrow X_n \leftarrow Y$, and }
\item{$X_i \to \Spec\,k_i\to \Spec\,k_i^0$,}
\end{itemize}
which satisfies the properties (1)--(7) in Proposition~\ref{variety-main}. 
%\begin{enumerate}
%\item{The field extension $k_n/k$ is a finite purely inseparable extension.}
%\item{For every $i$, $X_i$ is a normal variety over $k_i$, and 
%the morphism $X_{i+1} \to X_i$ is a finite surjective purely inseparable morphism.}
%\item{For every $i$, $k_i$ is the purely inseparable closure of $k_i^0$ in $K(X_i)$. }
%\item{For every $i$, $[k_{i+1}^0:k_i]=p$. }
%\item{For every $i$, 
%$X_i\times_{k_i} k_{i+1}^0$ is integral and $X_{i+1}$ is the normalization of $X_i\times_{k_i} k_{i+1}^0$.}
%\item{$X_n$ is geometrically reduced over $k_n$. }
%\item{The induced morphism $Y \to X_n\otimes_{k_n} k^{p^{-\infty}}$ 
%is the normalization of $X_n\otimes_{k_n} k^{p^{-\infty}}$. }
%\end{enumerate}
%Since we can omit codimension two closed subsets, 
%we may assume that every $X_i$ is regular. 
%In particular, each $X_i\times_{k_i} k_{i+1}^0$ is Gorenstein. 

From $X_{i} \to \Spec\, k_{i}$, 
the morphism $X_{i+1} \to \Spec\, k_{i+1}$ is obtained by the following three steps. 
\begin{enumerate}
\item[$(\alpha)$]{Take the base change $X_i\otimes_{k_i} k_{i+1}^0$ where 
$X_i\otimes_{k_i} k_{i+1}^0$ is integral.}
\item[$(\beta)$]{Take the normalization $X_{i+1}$ of $X_i\otimes_{k_i} k_{i+1}^0$.}
\item[$(\gamma)$]{Let $k_{i+1}$ be the purely inseparable closure of $k_{i+1}^0$ in $K(X_{i+1})$.}
\end{enumerate}
Let $\beta_{i}:X_i\otimes_{k_i} k_{i+1}^0 \to X_i$ be the projection and 
let $\nu_{i}:X_{i+1} \to X_i\otimes_{k_i} k_{i+1}^0$ be the normalization. 
Then, we can compare $K_{X_i}$ and $K_{X_{i+1}}$  as follows. 
\begin{enumerate}
\item[$(\alpha)$]{$\omega_{X_i\otimes_{k_i} k_{i+1}^0/k_{i+1}^0}\simeq \beta_{i}^*\omega_{X_i/k_i}$ 
holds outside a codimension two closed subset (Theorem~\ref{K-bc}). }
\item[$(\beta)$]{$\omega_{X_{i+1}/k_{i+1}^0}\otimes_{\MO_{X_{i+1}}} \MO_{X_{i+1}}(C_{i})\simeq \nu_{i}^*\omega_{X_i\otimes_{k_i} k_{i+1}^0/k_{i+1}^0}$ holds outside a codimension two closed subset where 
the support of $C_{i}$ is equal to the support of the conductor 
of $\nu_{i}$ (Proposition~\ref{K-normalization}).}
\item[$(\gamma)$]{$\omega_{X_{i+1}/k_{i+1}}\simeq \omega_{X_{i+1}/k_{i+1}^0}$ (Lemma~\ref{base-invariance}).}
\end{enumerate}
Note that, outside a codimension two locus, every $X_i$ is regular 
and every $X_i\otimes_{k_i} k_{i+1}^0$ is Gorenstein. 

Therefore, we see that 
$$K_{X^n}+C'=g^*K_X$$
for some effective $\Z$-divisor $C'$ where $g:X_n \to X$. 
Note that $C'$ is the sum of the pull-backs of $C_{1}, C_{2}, \cdots, C_{n-1}$. 
Let $h:Y\to X_n$ be the induced morphism. 
By (7) in Proposition~\ref{variety-main} and the proof of the case Step~1, we see $K_Y+C''=h^*K_{X_n}$ for some effective $\Z$-divisor $C''$. 
Note that $C''$ may be zero. 
Then we obtain 
$$K_Y+C''+h^*C'=h^*(K_{X_n}+C')=h^*(g^*K_X)=f^*K_X.$$
Set $C:=C''+h^*C'$. 
\end{proof}

\medskip
\textbf{Step 3:} 
If the condition (b) holds, then we may assume that $C \neq 0$.

\begin{proof}[Proof of Step~3]
We may assume that $X$ is proper, 
because $X$ is not geometrically reduced over $k$ if and only if 
$K(X)$ is not geometrically reduced over $k$. 
Since $k$ is algebraically closed in $K(X)$, 
we obtain $H^0(X, \MO_X)=k$. 
%\alpha_*\MO_X=\MO_{\Spec\,k}$ where $\alpha:X \to \Spec\,k$ is the structure morphism. 

We apply the same argument as Step~2 and use the same notation as Step~2. 
To show $C\neq 0$, we prove $C'\neq 0$. 
Assume the contrary, that is, assume $C_{i}=0$ for every $i$. 
Let us derive a contradiction. 
%it suffices to prove that 
%$\nu_i$ is an isomorphism and that $k_{i}=k_i^0$. 
The equation $C_{i}=0$ implies that every $X_i\otimes_{k_i} k_{i+1}^0$ is regular in codimension one. 
Thus, every $X_i\otimes_{k_i} k_{i+1}^0$ is normal. 
Thus $\nu_{i}$ is an isomorphism and 
$X_{i+1}=X_i\otimes_{k_i} k_{i+1}^0$. 

We show $H^0(X_i, \MO_{X_i})=k_i^0$ by the induction on $i$. 
%$$\alpha_i:X_i \overset{\gamma_i}\to \Spec\,k_i \to \Spec\,k_i^0.$$ 
When $i=1$, this holds by $H^0(X, \MO_X)=k$. 
Fix $i\geq 1$ and assume that $H^0(X_i, \MO_{X_i})=k_i^0$ holds for $i$. 
In particular, we see 
$H^0(X_i, \MO_{X_i})=k_i=k_i^0$. 
%Since $X_i$ is a proper normal variety, 
%$k_i^0$ is algebraically closed in $K(X_i)$ if and only if $\alpha_*\MO_{X_i}=\MO_{\Spec\,k_i^0}$. 
By the flat base change theorem, 
$H^0(X_i, \MO_{X_i})=k_i$ implies $H^0(X_{i+1}, \MO_{X_{i+1}})=k_{i+1}^0$. 
Thus we see $H^0(X_i, \MO_{X_i})=k_i^0$ for every $i$. 

Since $X_{i}$ is normal, $k_i^0$ is algebraically closed in $K(X_{i})$. 
This implies $k_i^0=k_i$. 
Then we see that $X_n\simeq X \otimes_{k} k_n$ and 
$X_n$ is geometrically reduced over $k_n$. 
Therefore, $X$ is geometrically reduced over $k$, which violates our assumption. 
\end{proof}
The assertion in the theorem holds. 
\end{proof}

%\begin{cor}
%Let $k$ be a field of characteristic $p>0$. 
%Let $X$ be a normal variety over $k$. 
%Set $Y$ to be the normalization of $(X\times_{k} k^{1/p^{\infty}})_{\red}$ 
%and let $f:Y \to X$ be the induced affine morphism. 
%Then, there exists an effective $\Z$-divisor $C$ on $Y$ such that 
%$$(\omega_{Y/k^{1/p^{\infty}}} \otimes_{\MO_Y}\MO_Y(C))|_{Y_{\reg}} \simeq f^*(\omega_{X/k}|_{X_{\reg}})|%_{Y_{\reg}},$$
%where $X_{\reg}$ and $Y_{\reg}$ are the regular loci of $X$ and $Y$, respectively. 
%\end{cor}

%\begin{proof}
%Let $k_1$ be the purely inseparable closure of $k$ in $K(X)$.  
%We can check that $(X\times_{k_1} k_1^{1/p^{\infty}}) \simeq (X\times_{k} k^{1/p^{\infty}})$. 
%Therefore, if $X$ is not geometrically normal, 
%then the assertion follows from Theorem~\ref{bc-main}. 
%If $X$ is geometrically normal, then we obtain 
%$\omega_{Y/k^{1/p^{\infty}}}|_{Y_{\reg}} \simeq f^*(\omega_{X/k}|_{X_{\reg}})|_{Y_{\reg}}$ 
%by Theorem~\ref{K-bc}. 
%\end{proof}

We give an explicit calculation for the double conic.

\begin{ex}\label{double-conic}
Let $F$ be an algebraically closed field of characteristic two. 
Set $k:=F(s, t)$ to be the purely transcendental extension of $F$ of degree two. 
Let 
$$X:=\Spec\,k[x, y]/(sx^2+ty^2+1).$$
Since $F[s, t, x, y](sx^2+ty^2+1)$ is smooth over $F$ 
by the Jacobian criterion for the smoothness, 
its localization $X$ is a regular curve. 
Set 
$$k':=k(s^{1/2})=F(s^{1/2}, t).$$
Take the base change
\begin{eqnarray*}
X\otimes_k k'
&=&\Spec\,k'[x, y]/(sx^2+ty^2+1)\\
&\simeq&\Spec\,k'[x, y]/(x^2+ty^2+1)\\
&\simeq&\Spec\,k'[x, y]/(x^2+ty^2).\\
\end{eqnarray*}
%where we applies the following transforms
%$$x'\mapsto s^{1/2}x,\,\,\, x \mapsto x+1.$$ 
Note that $X\otimes_k k'$ is integral but not normal because 
the origin is a unique non-regular point. 
Taking the blowup at the origin, 
we see that the normalization 
$X_1$ of $X\otimes_k k'$ is written by 
$$X_1=\Spec\,k'[x, y]/(x^2+t)\simeq \Spec\,k'(t^{1/2})[y] \simeq \mathbb A^1_{k'(t^{1/2})}.$$
We set 
$k_1:=k'(t^{1/2})=F(s^{1/2}, t^{1/2}).$
Therefore, $X_1$ is geometrically reduced over $k_1$. 
\end{ex}

In Theorem~\ref{bc-main}, 
the scheme $(X\otimes_{k} k^{1/p^{\infty}})_{\red}$ 
is not changed by replacing a larger base field $k'$ 
which is purely inseparable over $k$. 

\begin{lem}\label{purely-enlarging}
Let $k$ be a field of characteristic $p>0$ and 
let $k'/k$ be a finite purely inseparable extension. 
Let $X$ be a scheme of finite type over $k'$. 
Then, the natural morphism 
$$\theta:(X\otimes_{k'} k'^{p^{-\infty}})_{\red} \to (X\otimes_k k^{p^{-\infty}})_{\red}.$$ 
is an isomorphism. 
\end{lem}

\begin{proof}
Since $k'^{p^{-\infty}}=k^{p^{-\infty}}$, 
the assertion follows from the fact that 
$X\otimes_{k'} k'^{p^{-\infty}} \to X\otimes_k k^{p^{-\infty}}$ 
is a bijective closed immersion. 
%Since the problem is local, we can assume that $X=\Spec\,B$. 
%Set 
%$$\varphi:(B \otimes_k k^{p^{-\infty}})_{\red} \to (B \otimes_{k'} k'^{p^{-\infty}})_{\red}$$ 
%to be the ring homomorphism corresponding to $\theta$. 
%Clearly, $\varphi$ is surjective. 
%Thus, we show that $\varphi$ is injective. 
%To apply Lemma~\ref{tensor-inje}, 
%it suffices to check that 
%the natural homomorphism 
%$$\psi:B \to (B \otimes_{k'} k'^{p^{-\infty}})_{\red}$$
%is injective. 
%Since $B \to B \otimes_{k'} k'^{p^{-\infty}}$ is injective and $B$ is reduced, 
%$\psi$ is injective. 
%Therefore, also $\varphi$ is injective by Lemma~\ref{tensor-inje}. 
\end{proof}

%The following lemma is used in the proof of Lemma\ref{purely-enlarging}. 

%\begin{lem}\label{tensor-inje}
%Let $A\subset R$ be a ring extension of $\mathbb{F}_p$-algebras. 
%Let $A \subset A' \subset R$ and let $A\subset B \subset R$ be 
%two intermediate rings. 
%Assume that the ring extension $A\subset A'$ is purely inseparable, that is, 
%for every $\alpha \in A'$, there is $e\in\mathbb Z_{>0}$ such that $\alpha^{p^e}\in A$. 
%If $R$ is reduced, then the natural ring homomorphism 
%$$\theta:(B\otimes_A A')_{\red} \to R$$
%is injective. 
%\end{lem}

%\begin{proof}
%We have the following homomorphisms 
%$$B \overset{j}\to B\otimes_{A} A' \overset{\pi}\to (B\otimes_A A')_{\red} \overset{\theta}\to R.$$
%Since this composition map is injective, so is $j$. 
%Let $\xi \in B\otimes_{A} A'$. 
%Assume $\theta\circ \pi(\xi)=0$. 
%It suffices to show $\xi^{\ell}=0$ for some $\ell>0$. 
%Since $A\subset A'$ is purely inseparable, $\xi^{p^e}$ comes from $B$, that is, 
%there is $b\in B$ such that $\xi^{p^e}=j(b).$ 
%Since 
%$$\theta\circ \pi\circ j(b)=\theta\circ \pi(\xi^{p^e})=0$$
%and $\theta\circ \pi\circ j:B \to R$ is injective, 
%we obtain $b=0$. 
%Thus, $\xi^{p^e}=j(b)=0$. 
%This is what we want to show. 
%\end{proof}

%%%%%
\section{Uniruledness of bad $K_X$-trivial fibrations}

In this section, we show Theorem~\ref{uniruled-main}. 
We recall the definition of uniruled varieties.

\begin{dfn}[Ch IV, Definition 1.1 of \cite{Kollar1}]\label{def-uniruled}
Let $k$ be a field. 
Let $X$ be a variety over $k$. 
We say $X$ is {\em uniruled} 
if there exist a variety $Y$ of $\dim Y=\dim X-1$ 
and a dominant rational map 
$$Y \times_k \mathbb P^1_k \dashrightarrow X.$$
For a separated scheme $X$ of finite type over $k$, 
$X$ is {\em uniruled} if and only if every irreducible component of $X_{\red}$ is uniruled. 
\end{dfn}

The following lemma is fundamental. 

\begin{lem}\label{bc-uniruled}
Let $k \subset k'$ be a field extension. 
Let $X$ be a normal geometrically connected variety over $k$. 
Then, $X$ is uniruled if and only if the normalization $Y$ of $(X\otimes_k k')_{\red}$ 
is uniruled. 
\end{lem}

\begin{proof}
If $k$ is characteristic zero, then the assertion follows from \cite[Ch~IV, Proposition~1.3]{Kollar1}. 
Assume that $k$ is characteristic $p>0$. 
Again, by \cite[Ch~IV, Proposition~1.3]{Kollar1}, 
$X$ is uniruled if and only if $X \otimes_k k'$ is uniruled. 
By Definition~\ref{def-uniruled}, 
$X \otimes_k k'$ is uniruled if and only if $(X\otimes_k k')_{\red}$ is uniruled. 
Since the normalization $Y=(X\otimes_k k')_{\red}^N \to (X\otimes_k k')_{\red}$ is a universally homeomorphism 
by Lemma~\ref{geom-irreducible}(2), 
this factors through an iterated Frobenius 
$$F^e:(X\otimes_k k')_{\red} \to (X\otimes_k k')_{\red}.$$ 
Thus 
$(X\otimes_k k')_{\red}^N$ is uniruled if and only if $(X\otimes_k k')_{\red}$ is uniruled. 
\end{proof}

Thanks to the bend and break technique, 
we obtain a criterion for uniruledness. 

\begin{lem}\label{BB}
Let $k$ be a perfect field of characteristic $p>0$. 
Let $X$ be a projective normal variety over $k$. 
Assume that the equation $-K_X=N+E$ holds 
as Weil $\R$-divisors 
where $N$ is a nef $\R$-Cartier $\R$-divisor and 
$E$ is an effective $\R$-divisor. 
If $E\neq 0$ or $N\not\equiv 0$, then $X$ is uniruled.  
\end{lem}

\begin{proof}
We can replace $k$ by its algebraic closure $\overline k$. 
Indeed, the pullbacks of $N$ and $E$ by $X\otimes_k \overline k \to X$ 
are nef and effective, respectively. 
Thus we can assume that $k$ is algebraically closed. 
Then the assertion follows from the same proof as \cite[Ch~IV, Corollary~1.14]{Kollar1}. 
\end{proof}

The following theorem holds only in positive characteristic because 
a normal variety over a field of characteristic zero is geometrically normal. 

\begin{thm}\label{imperfect-uniruled}
Let $k$ be a field of characteristic $p>0$. 
Let $X$ be a projective normal variety over $k$ 
such that $H^0(X, \MO_X)=k$. 
Assume that $-K_X=N+E$ where $N$ is a nef $\R$-Cartier $\R$-divisor 
and $E$ is an effective $\R$-divisor. 
If $X$ is not geometrically normal over $k$, then $X$ is uniruled. 
%Moreover, assume that one of the following conditions hold. 
%\begin{enumerate}
%\item{.}
%\item{$X$ is geometrically reduced over $k$ but not geometrically normal over $k$.}
%\end{enumerate}
%Then $X$ is uniruled. 
\end{thm}

\begin{proof}
%Assume that (1) or (2) holds. 
Let $Y$ be the normalization of $(X\otimes_k k^{p^{-\infty}})_{\red}$ 
and let $f:Y \to X$ be the induced morphism. 
Then, by Theorem~\ref{bc-main}, 
we obtain 
$$K_Y+C=f^*K_X$$
for some nonzero effective $\Z$-divisor $C$.  
By Lemma~\ref{bc-uniruled}, it suffices to show that $Y$ is uniruled. 
Since $C\neq 0$, the assertion holds by Lemma~\ref{BB}. 
\end{proof}

We show the main theorem in this section. 

\begin{thm}\label{uniruled-main}
Let $k$ be a field of characteristic $p>0$. 
Let $\pi:X \to S$ be a projective surjective $k$-morphism of 
normal $k$-varieties such that $\pi_*\MO_X=\MO_S$. 
Assume that $-K_X=N+E$ 
where $N$ is a $\pi$-nef $\R$-Cartier $\R$-divisor and $E$ is an effective $\R$-divisor. 
If the generic fiber of $\pi$ is not normal, then $X$ is uniruled. 
%\begin{enumerate}
%\item{The generic fiber of $\pi$ is not geometrically reduced.}
%\item{The generic fiber of $\pi$ is geometrically reduced but not geometrically normal.}
%\end{enumerate}
%Then $X$ is uniruled. 
\end{thm}

\begin{proof}
Set $X_{K(S)}:=X\times_S K(S)$. 
To apply Theorem~\ref{imperfect-uniruled} for $X_{K(S)} \to \Spec\,K(S)$, 
we check that the conditions of Theorem~\ref{imperfect-uniruled} hold. 
Since $\Spec\,K(S) \to S$ is flat, we obtain 
$H^0(X_{K(S)}, \MO_{X_{K(S)}})=K(S)$. 
%Thus $K(S)$ is algebraically closed in $K(X_{K(S)})$. 
We obtain $-K_{X_{K(S)}}=N|_{X_{K(S)}}+E|_{X_{K(S)}}.$ 
We see that $N|_{X_{K(S)}}$ is a nef $\R$-Cartier $\R$-divisor and 
$E|_{X_{K(S)}}$ is an effective $\R$-divisor. 
Then we can apply Theorem~\ref{imperfect-uniruled} for $X_{K(S)} \to \Spec\,K(S)$, and 
we see that $X_{K(S)}$ is uniruled over $K(S)$. 
Thus we obtain a dominant rational map 
$$Y \times_{K(S)} \mathbb P^1_{K(S)} \dashrightarrow X_{K(S)}$$
where $Y$ is a $(\dim X-\dim~S-1)$-dimensional variety over $K(S)$. 
By killing denominators, we can find an open affine $k$-variety $\Spec\,R\subset S$ 
and a $(\dim X-1)$-dimensional variety $Y_{R}$ over $R$ 
which satisfies $Y_R \otimes_R K(S) \simeq Y$, equipped with a dominant rational map 
$$Y_{R} \times_{R} \mathbb P^1_{R} \dashrightarrow X_{R}$$
where $X_R:=\pi^{-1}(\Spec\,R)$. 
Since $Y_{R} \times_{R} \mathbb P^1_{R}\simeq Y_{R}\times_k \mathbb P^1_k$, 
$X$ is uniruled. 
\end{proof}

%%%%%
\section{RCC-ness of general fibers of Mori fiber spaces}

In this section, we show Theorem~\ref{MFS-RCC}. 
We recall the definition of RCC schemes and MRCC-fibrations. 
For basic properties of them, see \cite[Ch. IV., Sections 3 and 5]{Kollar1}.

\begin{dfn}[Ch. IV, Definition 3.2 of \cite{Kollar1}]\label{def-MRCC}
Let $k$ be a field. 
Let $X$ be a proper scheme over $k$. 
We say $X$ is {\em rationally chain connected} (for short, RCC) 
if there is a family of proper and connected algebraic curves $g:U \to Y$ whose 
geometric fibers have only rational components with cycle morphism $u:U \to X$ 
such that the canonical morphism $U \times_Y U \to X \times_k X$ 
is dominant. 
\end{dfn}

\begin{dfn}[Ch. IV, Definition 5.1 of \cite{Kollar1}]\label{def-MRCC}
Let $k$ be a field. 
Let $X$ be a proper normal variety over $k$. 
\begin{enumerate}
\item{We say $X\supset X^0 \overset{\pi}\to Z$ is an {\em RCC fibration} if 
$X^0$ is a non-empty open subset of $X$ and 
$\pi:X^0 \to Z$ is a proper surjective morphism 
such that the fibers are RCC schemes and that $\pi_*\MO_{X_0}=\MO_Z$. }
\item{We say $X\supset X^0 \to Z$ is an {\em MRCC fibration} if 
$X\supset X^0 \to Z$ is an RCC fibration and, for every RCC fibration 
$X\supset X^1 \to W$, there exists a dominant rational map 
$Z\dashrightarrow W$ such that the induced field extension $K(Z) \supset K(W)$ 
is commutative with $K(X) \supset K(Z)$ and $K(X) \supset K(W)$. }
\end{enumerate}
\end{dfn}

\begin{rem}
Let $X$ be a proper scheme over a field $k$. 
\begin{enumerate}
\item{For a field extension $k'/k$, 
$X$ is RCC if and only if so is $X \otimes_k k'$ (cf. Proposition~\ref{MRCC-bc}). }
\item{
If $k$ is an uncountable algebraically closed field, 
then $X$ is RCC if and only if every pair of closed points $x_1, x_2 \in X$ 
can be connected by a connected chain of rational curves (\cite[Ch. IV, Proposition 3.6]{Kollar1}). }
\end{enumerate}
\end{rem}

Fano varieties with $\rho(X)=1$ play a crucial role in the minimal model theory. 
The following theorem states that such varieties are rationally chain connected. 

\begin{thm}\label{Fano-RCC}
Let $k$ be a field. 
Let $X$ be a projective normal geometrically connected variety over $k$. 
Assume the following conditions. 
\begin{itemize}
\item{$X$ is $\Q$-factorial.}
\item{$-K_X$ is ample.}
\item{$\rho(X)=1$.}
\end{itemize}
Then $X$ is rationally chain connected over $k$. 
\end{thm}

\begin{proof}
We may assume that $H^0(X, \MO_X)=k$.

We reduce the proof to the case where $k$ is perfect. 
Set ${\rm char}\,k=:p>0$. 
Let $Y$ be the normalization of $(X\otimes_k k^{p^{-\infty}})_{\red}$. 
Then, we see that $Y$ is $\Q$-factorial (Lemma~\ref{p-insep-qfac}) and 
$\rho(Y)=1$ (Lemma~\ref{p-insep-picard}). 
Moreover, by Theorem~\ref{bc-main}, 
we see 
$$K_Y+C=f^*K_X$$
for some effective divisor $C$ on $Y$. 
Since $\rho(Y)=1$, $C$ is nef. 
In particular, $-K_Y$ is ample. 
Thus, by replacing $X$ with $Y$, 
we may assume that $k$ is perfect. 

Set $Z:=X\otimes_k \overline k$. 
Note that $-K_Z$ is an ample $\Q$-Cartier divisor. 
Thus $Z$ is uniruled. 
Therefore, an MRCC fibration of $Z \supset Z^0 \to W$ satisfies $\dim W<\dim Z$. 
By Lemma~\ref{MRCC-bc}, 
an MRCC fibration 
$$X \supset X^0 \overset{\pi}\to V$$ 
also satisfies $\dim V<\dim X$. 
It suffices to show $\dim V=0$. 
Assume $\dim V>0$ and let us derive a contradiction. 
Fix a closed point $v\in V$. 
We can find an effective Cartier divisor $D_V$ on $V$ with $v\not \in D_V$ 
by shrinking $V$ if necessary. 
Take a proper curve $C_X \subset \pi^{-1}(v)$ and 
set $D_X$ to be the closure of $\pi^{-1}(D_V)$. 
Then, $D_X \cap C_X =\emptyset$ and 
this implies $D_X \cdot C_X=0$. 
Note that, since $X$ is $\Q$-factorial, $D_X$ is $\Q$-Cartier. 
This contradicts $\rho(X)=1$. 
\end{proof}

\begin{cor}\label{Fano-torsion}
Let $k$ be a field. 
Let $X$ be a projective normal variety over $k$. 
Assume the following conditions. 
\begin{enumerate}
\item{$X$ is $\Q$-factorial.}
\item{$-K_X$ is ample.}
\item{$\rho(X)=1$.}
\end{enumerate}
Then 
the torsion subgroup of ${\rm Pic}(X)$ coincides with the subgroup of numerically trivial classes.
\end{cor}

\begin{proof}
We may assume that $H^0(X, \MO_X)=k$. 
By the same argument as Theorem~\ref{Fano-RCC}, 
we may assume that $k$ is a perfect field. 
Take the base change to the algebraic closure $X\otimes_k \overline k$. 
By the assumption, $X\otimes_k \overline k$ is rationally chain connected. 
Then, the Albanese variety of $X\otimes_k \overline k$ is one point. 
Therefore, $D$ is torsion. 
\end{proof}

The following lemma is useful to compare the total space and the generic fibers.  

\begin{lem}\label{generic-Qfac}
Let $k$ be a field. 
Let $\pi:X \to S$ be a proper surjective $k$-morphism 
of normal $k$-varieties such that $\pi_*\MO_X=\MO_S$. 
\begin{enumerate}
\item{If $X$ is $\Q$-factorial, then the generic fiber $X_{K(S)}$ is also $\Q$-factorial. }
\item{If $X$ is $\Q$-factorial and $\rho(X/S)=1$, then $\rho(X_{K(S)})=1$.}
\end{enumerate}
\end{lem}

\begin{proof}
(1) 
The assertion follows from the fact that 
$X_{K(S)}$ is locally obtained as a localization of $X$. 
%Let $D$ be a prime divisor on $X_{K(S)}$. 
%Then we can find a non-empty open subset $S'\subset S$ 
%and a prime divisor $D_{S'}$ on $X':=\pi^{-1}(S')$ such that 
%$D_{S'}\times_{S'} \Spec\,K(S)=D$. 
%Since $X$ is $\Q$-factorial, 
%$mD_{S'}$ is Cartier for some $m\in\mathbb Z_{>0}$. 
%Therefore, its pull-back $mD$ to $X_{K(S)}$ is also Cartier. 

(2) 
Fix a Cartier divisor $D$ on $X_{K(S)}$ and a curve $C$ on $X_{K(S)}$. 
Assume $D\cdot C=0$. 
It suffices to show $D\equiv 0$. 
For this, fix an arbitrary curve $B$ on $X_{K(S)}$ and we prove $D\cdot B=0$. 
We can find a non-empty open subset $S'\subset S$, 
a Cartier divisor $D_{S'}$ on $\pi^{-1}(S')$ 
and flat families of curves $C_{S'}, B_{S'} \subset \pi^{-1}(S')$ over $S'$, whose pull-back and inverse images 
to $X_{K(S)}$ are $D, C, B$, respectively. 
Fix a closed point $s\in S'$ and let $D_s, C_s, B_s$ be the pull-back and inverse images to $X_s$. 
By the flatness of $C_{S'} \to S'$, 
we obtain 
$$D_s\cdot C_s=D_{S'}\cdot C_s=D\cdot C=0.$$ 
Let $D_S$ be the closure of $D_{S'}$ in $X$. 
Since $X$ is $\Q$-factorial, 
$D_S$ is $\Q$-Cartier. 
Then $\rho(X/S)=1$ and $D_S\cdot C_s=D_s\cdot C_s=0$ imply $D_S\equiv_{\pi} 0$. 
Therefore, $D_S\cdot B_s=0$. 
By the flatness of $B_{S'} \to S'$, we see 
$$D\cdot B=D_{S}\cdot B_s=0.$$ 
This is what we want to show. 
\end{proof}

We show the main theorem in this section. 

\begin{thm}\label{MFS-RCC}
Let $k$ be a field. 
Let $\pi:X \to S$ be a projective surjective $k$-morphism of 
normal $k$-varieties such that $\pi_*\MO_X=\MO_S$. 
Assume the following conditions. 
\begin{enumerate}
\item{$X$ is $\Q$-factorial.}
\item{$-K_X$ is $\pi$-ample.}
\item{$\rho(X/S)=1$.}
\end{enumerate}
Then general fibers of $\pi$ are rationally chain connected, 
i.e. there exists a non-empty open subset $S' \subset S$ such that 
for every point $s' \in S'$, the fiber $X_{s'}$ is rationally chain connected. 
\end{thm}

\begin{proof}
By \cite[Ch IV, Corollary 3.5]{Kollar1}, 
it suffices to show that 
the generic fiber is rationally chain connected. 
For this, we check that the generic fiber $X_{K(S)}$ satisfies the assumptions of Theorem~\ref{Fano-RCC}. 
By Lemma~\ref{generic-Qfac}, $X_{K(S)}$ is $\Q$-factorial and $\rho(X_{K(S)})=1$. 
Since $-K_X$ is $\pi$-ample, $-K_{X_{K(S)}}$ is ample. 
\end{proof}

%%%%%
\section{Cone theorem for surfaces and threefolds}

In this section, we show Theorem~\ref{surface-cone} and Theorem~\ref{weak-cone}. 
We recall the definition of nef thresholds. 

\begin{dfn}
Let $k$ be a field. 
Let $X$ be a projective normal variety over $k$. 
Let $D$ be an $\R$-Cartier $\R$-divisor and 
let $H$ be an ample $\R$-Cartier $\R$-divisor. 
We define a {\em nef threshold} $a_H$ of $D$ with respect to $H$ by 
$$a_H:=\inf\{a\in\mathbb R_{\geq 0}\,|\, D+aH{\rm\,\,is\,\,nef}\}.$$
\end{dfn}

\begin{rem}\label{rem-nef-thre}
Let $k \subset k'$ be a field extension. 
Let $X$ be a projective normal variety over $k$. 
Set $f:Y \to X$ be the normalization of $(X \otimes_k k')_{\red}$. 
%Set $\beta:X\otimes_k k' \to X$ to be the projection. 
Let $D$ be an $\R$-Cartier $\R$-divisor and 
let $H$ be an ample $\R$-Cartier $\R$-divisor. 
Then, by Lemma~\ref{intersection-bc}, 
the nef threshold of $D$ with respect to $H$ 
is equal to the nef threshold of $f^*D$ with respect to $f^*H$. 
\end{rem}

Recall that for a $K_X$-negative extremal ray $R$ of $\overline{NE}(X)$, 
its {\em length} is defined by 
$$\min\{-K_X \cdot C\,|\, C{\rm \,\,is\,\,a\,\,curve\,\,such\,\,that}\,\,[C] \in R\}.$$
One of the main strategy to show the cone theorem 
is to find an upper bound of length of extremal rays. 
However it may be difficult to find such a bound by the following example. 

\begin{ex}\label{length-infinite}
Fix a positive integer $N \in \mathbb Z_{>0}$. 
Let $k$ be a field such that $[\overline k:k]=\infty.$ 
Let $Y:=\mathbb P_k^2$. 
Then, we can find a closed point $y\in Y$ 
such that $[k(y):k]>N$. 
Take the blowup at $y$: $f:X \to Y$ 
and let $C$ be the $f$-exceptional curve. 
We see $-K_X\cdot C>0.$ 
Note that the structure morphism $C \to \Spec\,k$ 
factors through $C \to \Spec\,k(y) \to \Spec\,k$. 
Thus the residue field of every closed point of $C$ contains $k(y)$. 
Thus, $-K_X\cdot C$ is divisible by $[k(y):k]$. 
This implies $-K_X\cdot C>N.$

To summarize, for every positive integer $N\in \mathbb Z_{>0}$, 
there exist a projective regular surface $X$ over $k$ 
and a curve $C$ on $X$ which satisfy the following properties. 
\begin{enumerate}
\item{$\mathbb R_{\geq 0}[C] \subset \overline{NE}(X)$ is an extremal ray 
of $\overline{NE}(X)$.}
\item{$-K_X\cdot C>N.$}
\item{$C^2<0$.}
\end{enumerate}
\end{ex}

Instead of the boundedness of length of extremal rays, 
we use the following lemma. 

\begin{lem}\label{lemma-cone} 
Let $k$ be a field. 
Let $X$ be a projective normal variety over $k$. 
Let $\Delta$ be an $\mathbb R$-divisor on $X$ such that $K_X+\Delta$ is $\R$-Cartier. 
Let $A$ be an ample $\R$-Cartier $\mathbb{R}$-divisor on $X$. 
For any ample $\R$-Cartier $\mathbb R$-divisor $H$, 
let $a_H$ be the nef threshold of $K_X+\Delta+\frac 1 2 A$ with respect to $H$.

Assume that there exist finitely many curves $C_1, \cdots, C_m$, 
such that for any ample $\R$-Cartier $\R$-divisor $H$ on $X$, we have that 
$(K_X+\Delta+\frac{1}{2}A+a_HH) \cdot C_i=0$ for some $i$. 
Then 
$$\overline{NE}(X)=\overline{NE}(X)_{K_X+\Delta+A\ge 0}+\sum_{i=1}^m \mathbb R_{\ge 0}[C_i].$$
\end{lem}

\begin{proof}
We can apply the same proof as \cite[Lemma~6.2]{CTX}. 
\end{proof}

We show a cone theorem for surfaces. 

\begin{thm}\label{surface-cone}
Let $k$ be a field of characteristic $p>0$. 
Let $X$ be a projective normal surface over $k$ and 
let $\Delta$ be an effective $\R$-divisor such that $K_X+\Delta$ is $\R$-Cartier. 
Let $A$ be an ample $\R$-Cartier $\R$-divisor. 
\begin{enumerate}
\item{For any ample $\R$-Cartier $\mathbb R$-divisor $H$, 
let $a_H$ be the nef threshold of $K_X+\Delta+A$ with respect to $H$. 
Then, there exist finitely many curves $C_1, \cdots, C_m$, 
such that for any ample $\R$-Cartier $\mathbb R$-divisor $H$ on $X$, we have that 
$(K_X+\Delta+A+a_HH)\cdot C_i=0$ for some $i$. }
\item{There exists a finitely many curves $C_1, \cdots, C_r$ such that 
$$\overline{NE}(X)=\overline{NE}(X)_{K_X+\Delta+A \geq 0}+\sum_{i=1}^r \mathbb R_{\geq 0}[C_i].$$}
\end{enumerate}
\end{thm}

\begin{proof}
Since (2) follows from (1) by Lemma~\ref{lemma-cone}, 
we prove (1). 
%For the structure morphism $\alpha:X \to \Spec\,k$, 
We may assume $H^0(X, \MO_X)=k$ by taking the Stein factorization. 

Let $\overline k$ be the algebraic closure of $k$. 
Let $Y$ be the normalization of $(X\otimes_k \overline{k})_{\red}$. 
Then, $Y$ is a normal variety over $\overline k$. 
Let $f:Y \to X$ be the induced morphism. 
Then, by Theorem~\ref{bc-main}, we obtain 
$$K_Y+D=f^*K_X$$
for some effective divisor $D$. 
Thus, we can find an effective $\R$-divisor $\Delta_Y$ such that  
$$K_Y+\Delta_Y=f^*(K_X+\Delta)$$
and that $K_Y+\Delta_Y$ is $\R$-Cartier. 
Thus, $(\overline k, Y, \Delta_Y, f^*A)$ satisfies the same assumptions as $(k, X, \Delta, A)$. 
Note that, if $k$ is an algebraically closed field, 
then the assertion follows from \cite[Theorem~3.13(2)]{T}. 
Then, there exist curves $C_1', \cdots, C_m'$ on $Y$ 
such that for any ample $\R$-Cartier $\mathbb R$-divisor $H'$ on $Y$, we have that 
$(K_Y+\Delta_Y+f^*A+b_{H'}H')\cdot C_i'=0$ for some $i$, 
where $b_{H'}$ is the nef threshold of $K_Y+\Delta_Y+f^*A$ 
with respect to $H'$. 
We can find curves $C_1, \cdots, C_m$ on $X$ such that $C_i' \subset f^{-1}(C_i)$. 
Take an ample $\R$-Cartier $\R$-divisor $H$ on $X$. 
We see $a_H=b_{f^*H}$ (cf. Remark~\ref{rem-nef-thre}). 
Thus we obtain 
$$(K_Y+\Delta_Y+f^*A+a_{H}f^*H)\cdot C_i'=0$$ 
for some $i$. 
By Lemma~\ref{intersection-bc}, this implies 
$$(K_{X}+\Delta+A+a_{H}H)\cdot C_i=0.$$ 
\end{proof}

By the cone theorem for surfaces (Theorem~\ref{surface-cone}), 
we obtain a cone theorem for threefolds with pseudo-effective canonical divisors. 

\begin{thm}\label{threefold-effective-cone}
Let $k$ be a field of characteristic $p>0$. 
Let $X$ be a projective normal $\Q$-factorial threefold and 
let $\Delta$ be an effective $\R$-divisor. 
Assume the following two conditions. 
\begin{itemize}
\item{$0\leq \Delta \leq 1$.}
\item{$K_X+\Delta$ is pseudo-effective}
\end{itemize}
Let $A$ be an ample $\R$-divisor.
\begin{enumerate}
\item{For any ample $\mathbb R$-divisor $H$, 
let $a_H$ be the nef threshold of $K_X+\Delta+A$ with respect to $H$. 
Then, there exist finitely many curves $C_1$,..., $C_m$, 
such that for any ample $\mathbb R$-divisor $H$ on $X$, we have that 
$(K_X+\Delta+A+a_HH)\cdot C_i=0$ for some $i$. }
\item{There exists a finitely many curves $C_1, \cdots, C_r$ such that 
$$\overline{NE}(X)=\overline{NE}(X)_{K_X+\Delta+A \geq 0}+\sum_{i=1}^r \mathbb R_{\geq 0}[C_i].$$}
\end{enumerate}
\end{thm}

\begin{proof}
Since (2) follows from (1) by Lemma~\ref{lemma-cone}, 
we prove (1). 

Since $A$ is ample, we can assume $0\leq \Delta < 1$ and 
$K_X+\Delta\equiv E$ where $E$ is an effective $\R$-divisor. 
Set $E:=\sum_{i=1}^{s} e_iE_i$ to be the prime decomposition with $e_i>0$. 
We define $\lambda_i\in\mathbb R_{>0}$ by  
$$K_X+\Delta+\lambda_iE=K_X+E_i+D_i,$$
where $D_i$ is an effective $\R$-divisor with $E_i \not\subset\Supp D_i$. 
We obtain 
$$(1+\lambda_i)(K_X+\Delta)\equiv K_X+\Delta+\lambda_iE=K_X+E_i+D_i.$$
Set $\nu_i:E_i^N \to E_i$ to be the normalization. 
We define an effective $\R$-divisor $\Delta_{E^N_i}$ by the adjunction (cf. \cite[Proposition~4.5]{Kollar2}) 
$$(K_X+E_i+D_i)|_{E_i^N}=K_{E_i^N}+\Delta_{E^N_i}.$$
Let $\Gamma_1, \cdots, \Gamma_r$ be the curves in $X$ such that 
the support of $\bigcup_{i=1}^r \Gamma_j$ coincides with 
the dimension one part of $\Ex(\nu_1) \cup \cdots \cup \Ex(\nu_s)$. 
%each $\Gamma_j$ is contained in the conductors of 
%one of the normalizations $\nu_1:E_1^N \to E_1, \cdots, \nu_s:E_s^N \to E_s$. 
Let $b_{1, H}, \cdots, b_{s, H}$ be the nef thresholds of $(K_X+\Delta+A)|_{E_i^N}$ with respect to $H|_{E_i^N}$. 
We define $\gamma_j \in \R$ by $(K_X+\Delta+A+\gamma_j H)\cdot \Gamma_j=0$.

Set $\mu:=\max\{\max_{1\leq i\leq s} \{b_{i, H}\}, \max_{1\leq j\leq r} \{\gamma_j\}\}$ 
and we show that $a_H=\mu$. 
Since $K_X+\Delta+A+a_HH$ is nef, 
$(K_X+\Delta+A+a_HH)|_{E_i^N}$ is nef for every $i$ and 
$(K_X+\Delta+A+a_HH) \cdot \Gamma_j \geq 0$ for every $j$. 
Thus, we obtain $a_H \geq \mu$. 
We show the inverse inequality $a_H \leq \mu$. 
Assume that $(K_X+\Delta+A+cH)\cdot C<0$ for some $c\geq 0$ and curve $C$. 
We show that 
one of $\{(K_X+\Delta+A+cH)|_{E_i^N}\}_i$ is not nef or 
one of $\{(K_X+\Delta+A+cH) \cdot \Gamma_j\}_j$ is negative. 
Thus, we may assume that 
$(K_X+\Delta+A+cH)\cdot \Gamma_j\geq 0$ for every $j$. 
In particular, $C \neq \Gamma_j$ for every $j$.  
Since 
$$0>(K_X+\Delta+A+cH)\cdot C >(\sum_{i=1}^{s} e_iE_i) \cdot C,$$
we obtain $E_i\cdot C<0$, which implies $C \subset E_i$. 
Since $C \neq \Gamma_j$, we can take the proper transform $C' \subset E_i^N$ of $C$, 
which implies
$$0>(K_X+\Delta+A+cH)\cdot C=(K_X+\Delta+A+cH)|_{E_i}\cdot C=(K_X+\Delta+A+cH)|_{E_i^N}\cdot C'.$$
Thus $(K_X+\Delta+A+cH)|_{E_i^N}$ is not nef and we obtain $a_H \leq \mu$.

We have 
$$(K_X+\Delta+A+b_{i, H}H)|_{E_i^N} \equiv \frac{1}{1+\lambda_i}(K_{E_i^N}+\Delta_{E_i^N}+(1+\lambda_i)A+b_{i, H}(1+\lambda_i)H)|_{E_i^N}.$$
We apply Theorem~\ref{surface-cone} for $(E^N_i, \Delta_{E_i^N}; (1+\lambda_i)A|_{E_i^N}, (1+\lambda_i)H)$ and 
obtain curves $D_1^{(i)}, \cdots, D_{t_i}^{(i)}$ in $E^N_i$, 
which satisfy the property Theorem~\ref{surface-cone}(1). 

We show that 
the curves $\Gamma_1, \cdots, \Gamma_r$ and 
$\{\nu_i(D_1^{(i)}), \cdots, \nu_i(D_{t_i}^{(i)})\}_{1\leq i\leq s}$ satisfy the required properties. 
Take an ample $\R$-Cartier $\R$-divisor $H$. 
Since $a_H=\mu$, 
we obtain $a_H=b_{i, H}$ for some $i$ or $a_H=\gamma_j$ for some $j$. 
If $a_H=\gamma_j$, then there is nothing to show. 
We can assume $a_H=b_{i, H}$ for some $i$. 
Then, by the choice of $D_{\ell}^{(i)}$, 
we can find $D_{\ell}^{(i)}$ such that 
$(K_X+\Delta+A+b_{i, H}H)\cdot D_{\ell}^{(i)}=0$ 
for some $\ell$. 
We are done. 
\end{proof}

By a result of \cite{CTX}, we obtain a weak cone theorem for threefolds 
over an arbitrary field of positive characteristic. 

\begin{thm}\label{weak-cone}
Let $k$ be a field of characteristic $p>0$. 
Let $X$ be a projective normal $\Q$-factorial threefold over $k$ and 
let $\Delta$ be an effective $\R$-divisor such that $0\leq \Delta \leq 1$. 
If $K_X+\Delta$ is not nef, then there exist an ample $\Q$-divisor $A$ and 
finitely many curves $C_1, \cdots, C_m$ such that 
$K_X+\Delta+A$ is not nef and that 
$$\overline{NE}(X)=\overline{NE}(X)_{K_X+\Delta+A \geq 0}+\sum_{i=1}^m \mathbb R_{\geq 0} [C_i].$$
\end{thm}

\begin{proof}
Fix an ample $\Q$-divisor $A_0$ such that $K_X+\Delta+A_0$ is not nef. 
Set $a_H$ to be the nef threshold of $K_X+\Delta+A_0$ with respect to $H$. 
There are the following two cases. 
\begin{enumerate}
\item{There exists an ample $\R$-divisor $H$ such that 
$K_X+\Delta+A_0+a_HH$ is big.}
\item{For every ample $\R$-divisor $H$, $K_X+\Delta+A_0+a_HH$ is not big.}
\end{enumerate}

(1) Assume that there exists an ample $\R$-divisor $H$ such that 
$K_X+\Delta+A_0+a_HH$ is big. 
Then, $K_X+\Delta+A_0+(a_H-\epsilon)H$ is also big but not nef 
for a small positive real number $\epsilon>0$. 
Then, by perturbing $A_0+(a_H-\epsilon)H$, we can find an ample $\Q$-divisor $A$ 
such that $K_X+\Delta+A$ is big but not nef. 
Then, the assertion follows from Theorem~\ref{threefold-effective-cone}. 

(2) 
In this case, we set $A:=A_0$ and we show that 
$A$ satisfies the required properties. 
For every ample $\R$-divisor $H$, $K_X+\Delta+A+a_HH$ is not big. 
We may assume that $H^0(X, \MO_X)=k$. 
Set $Y$ to be the normalization of $(X \times_k \overline k)_{\red}$ 
where $\overline k$ is the algebraic closure of $k$. 
By Theorem~\ref{bc-main}, we can write $K_{Y}+D=f^*K_X$ 
for some effective divisor $D$ 
where $f:Y \to X$ is the induced morphism. 
Then, we obtain $K_Y+\Delta_Y=f^*(K_X+\Delta)$ 
for some effective $\R$-divisor $\Delta_Y$. 
Then, by \cite[Lemma~5.3]{CTX}, 
we can find a curve $C'$ on $Y$ such that 
$$(K_Y+\Delta_Y+f^*(A+a_HH))\cdot C'=0$$
and that 
$$f^*A \cdot C' \leq -(K_Y+\Delta_Y) \cdot C' \leq 2\dim X.$$ 
Thus, 
we can find curves $C_1', \cdots, C_m'$ on $Y$ 
such that 
for every ample $\R$-divisor $H$, 
we obtain 
$$(K_Y+\Delta_Y+f^*(A+a_HH))\cdot C'_i=0$$
for some $C_i'$. 
Let $C_i$ be the curve such that $C_i' \subset f^{-1}(C_i)$. 
Then, by Lemma~\ref{intersection-bc}, 
for every ample $\R$-divisor $H$, 
we can find $C_i$ such that 
$$(K_X+\Delta+A+a_H H)\cdot C_i=0.$$
By Lemma~\ref{lemma-cone}, we obtain the assertion. 
\end{proof}

\appendix
\def\thesection{A}
\section{MRCC fibrations and base changes}

The purpose of this section is to show Proposition~\ref{MRCC-bc}. 
This result is intrinsically shown in the proof of \cite[Theorem~5.2]{Kollar1}, 
i.e. it depends on \cite[Theorem~4.17]{Kollar1} and by using the property (4.17.2) in \cite[Theorem~4.17]{Kollar1}, 
we can deduce Proposition~\ref{MRCC-bc}. 
However, we give a proof of it for the sake of completeness. 
First, we establish two lemmas. 
%we get the functoriality of MRCC fibration by (4.17.2), which implies A.4. 

\begin{lem}\label{quotient-fibration}
Let $k$ be a field. 
Let 
$$\begin{CD}
Y @>f>> X\\
@VV\pi V\\
T
\end{CD}$$
be $k$-morphisms of $k$-varieties which satisfy the following properties. 
\begin{itemize}
\item{$X, Y,$ and $T$ are normal $k$-varieties. }
\item{$T$ is an affine scheme.}
\item{$\pi$ is a proper surjective morphism such that $\pi_*\MO_Y=\MO_T$ and 
$f$ is a finite surjective morphism. }
\end{itemize}
Then $S:=\Spec\,H^0(X, \MO_X)$ completes the following commutative diagram 
$$\begin{CD}
Y @>f>> X\\
@VV\pi V @VV\rho V\\
T @>g>> S, 
\end{CD}$$
such that $S$ is a normal $k$-variety, 
where $\rho$ is a proper surjective $k$-morphism such that $\rho_*\MO_X=\MO_S$ and 
$g$ is a finite surjective $k$-morphism. 
\end{lem}

\begin{proof}
Fix an affine open cover $X=\bigcup_{i \in I} X_i$ and set $Y_i:=f^{-1}(X_i)$. 
%We set 
%$$R_{X_i}:=\Gamma(X_i, \MO_X),\,\,\, R_{Y_i}:=\Gamma(Y_i, \MO_Y).$$
Clearly, $S$ satisfies the commutative diagram in the lemma. 
%Set $R_S:=\Gamma(X, \MO_X)$ and $R_T:=\Gamma(T, \MO_T)$. %and 
%$S:=\Spec\,R_S$. 
Note that 
$$\Gamma(S, \MO_S)=\Gamma(X, \MO_X)=\bigcap_{x\in X}\MO_{X, x}=\bigcap_{i \in I}\Gamma(X_i, \MO_X).$$
$$\Gamma(T, \MO_T)=\Gamma(Y, \MO_Y)=\bigcap_{y\in Y}\MO_{Y, y}=\bigcap_{i \in I}\Gamma(Y_i, \MO_Y).$$
Thus $S$ is an affine integral $k$-scheme. 
We show that $S$ satisfies the required properties in the lemma. 
For this, it suffices to show that $g$ is a finite morphism by the Eakin--Nagata theorem. 
Taking the separable closure of $K(Y)/K(X)$, 
we may assume that either $g$ is separable or purely inseparable. 

Suppose that $K(Y)/K(X)$ is a separable. 
Let $L\supset K(Y) \supset K(X)$ be the Galois closure and 
set $G:={\rm Gal}(L/K(X))$. 
Let $g:Z \to Y$ be the normalization of $Y$ in $L$ and set $Z_i:=g^{-1}(Y_i)$. 
Since the composite morphism $Z \to Y \to T$ is proper, 
the ring $\Gamma(Z, \MO_Z)$ is a finitely generated $k$-algebra. 
We obtain 
$$\Gamma(S, \MO_S)= \bigcap_{i \in I}\Gamma(X_i, \MO_X)=\bigcap_{i \in I} (\Gamma(Z_i, \MO_Z)^G)
=\left(\bigcap_{i \in I} \Gamma(Z_i, \MO_Z)\right)^G=\Gamma(Z, \MO_Z)^G.$$
Therefore, $\Gamma(S, \MO_S)$ is a finitely generated $k$-algebra and 
$\Gamma(Z, \MO_Z)$ is a finitely generated $\Gamma(S, \MO_S)$-module, 
hence so is $\Gamma(Y, \MO_Y)=\Gamma(T, \MO_T)$. 
Thus, $g$ is finite. 

Therefore, we may assume that $K(Y)/K(X)$ is purely inseparable. 
We can find $e \in \Z_{>0}$ such that $K(X) \supset K(Y)^{p^e}$, 
in particular, $\Gamma(X_i, \MO_X) \supset \Gamma(Y_i, \MO_Y)^{p^e}$. 
%where $k(K(Y)^{p^e})$ is the minimum subfield of $K(Y)$ containing $k$ and $K(Y)^{p^e}$. 
%Note that $[K(Y):k(K(Y)^{p^e})]<\infty$. 
We have 
$$\Gamma(S, \MO_S)= \bigcap_{i \in I}\Gamma(X_i, \MO_X) 
\supset k\left[\bigcap_{i \in I}\left(\Gamma(Y_i, \MO_Y)^{p^e}\right)\right]$$
$$=k\left[\left(\bigcap_{i \in I}\Gamma(Y_i, \MO_Y)\right)^{p^e}\right]=k\left[\Gamma(T, \MO_T)^{p^e}\right],$$
where $k[A]$ means the minimum $k$-algebra containing $A$. 
Here $\Gamma(T, \MO_T)$ is a finitely generated $k\left[\Gamma(T, \MO_T)^{p^e}\right]$-module 
because it is an integral extension. 
Therefore, $\Gamma(T, \MO_T)$ is a finitely generated $\Gamma(S, \MO_S)$-module. 
We are done. 
\end{proof}

\begin{rem}
Lemma~\ref{quotient-fibration} fails when $T$ is not affine. 
Actually, there is a finite surjective morphism
$$Y:=\mathbb P^1 \times \mathbb P^1 \to \mathbb P^2=:X.$$
On the other hand, $Y$ has a proper morphism to a curve but $X$ does not. 
\end{rem}

\begin{lem}\label{rational-map}
Let $k$ be a field. 
Let $Y$ be a proper normal $k$-variety. 
Let $Y^0\subset Y$ be a non-empty open subset and 
let 
$$\pi:Y^0 \to Z=\Spec\,R$$ 
be a proper surjective $k$-morphism to an affine $k$-variety $Z$ 
such that $\pi_*\MO_{Y^0}=\MO_Z$. 
We fix an embedding $K(Z)\subset K(Y)$ induced by $\pi$. 
Then 
$$Y^0=\{y\in Y\,|\, R\subset \MO_{Y, y}\}.$$
\end{lem}

\begin{proof}
Set $Y^1:=\{y\in Y\,|\, R\subset \MO_{Y, y}\}.$ 
We show $Y^0 \subset Y^1$. 
Take $y\in Y^0$. 
Then, we obtain 
$$R=\Gamma(Z, \MO_Z)=\Gamma(Z, \pi_*\MO_{Y^0})
=\Gamma(Y^0, \MO_{Y^0})\subset \MO_{Y, y}.$$
This implies $y\in Y^1$. 

We prove the inverse inclusion $Y^0 \supset Y^1$. 
We take 
%$y\in Y\setminus Y^0$ and assume 
$y\in Y^1$. 
%Let us derive a contradiction. 
Then, we obtain the following commutative diagram of inclusions 
$$\begin{CD}
K(Z) @>>>K(Y)\\
@AAA @AAA\\
R @>>> \MO_{Y, y}.
\end{CD}$$
Take a normal projective compactification $Z \subset \overline Z$ and 
the normalization of 
a blowup of the indeterminacy of $Y \dasharrow \overline Z$, denoted by $\pi':Y' \to \overline Z$. 
We obtain 
$$\begin{CD}
Y^0 @>{\rm open}>{\rm immersion}>Y'\\
@VV\pi V @VV\pi'V\\
Z @>{\rm open}>{\rm immersion}> \overline Z.
\end{CD}$$
Fix a point $y'$ over $y$. 
Then, we obtain $R \subset \MO_{Y', y'}$. 
Therefore, $\pi'(y')$ is a point $z$ in $Z=\Spec\,R$. 
It suffices to show that $\pi^{-1}(z)=\pi'^{-1}(z)$. 
Note that $\pi^{-1}(z)$ is proper over $k$ and that 
$\pi^{-1}(z)=\pi'^{-1}(z) \cap Y^0$ is an open subset of $\pi'^{-1}(z)$. 
Since $\pi_*\MO_{Y^0}=\MO_Z$, the field $K(Z)$ is algebraically closed in $K(Y)$. 
Therefore, $\pi'_*\MO_{Y'}=\MO_{\overline Z}$. 
In particular, $\pi'^{-1}(z)$ is connected, which implies $\pi^{-1}(z)=\pi'^{-1}(z)$. 
Indeed, otherwise $\pi^{-1}(z)=\pi'^{-1}(z) \cap Y^0$ is not proper over $k$. 
\end{proof}

We prove the main result of this section. 

%\begin{lem}
%Let $k \subset k'$ be a field extension. 
%Let $X$ be a proper geometrically normal and geometrically connected variety over $k$. 
%Let $X \supset X^0 \to V$ and $X\times_k k' \supset Y^0 \to W$ 
%be MRCC fibrations of $X$ and $X\times_k k'$, respectively. 
%Then, $\dim V=\dim W$. 
%\end{lem}

%\begin{proof}
%Take the base change $X\times_k k' \supset X^0\times_k k' \to V\times_k k'$. 
%Then, this is an RCC fibration. 
%Thus, $\dim V \leq \dim W$. 
%We show the inverse inequality $\dim V \geq \dim W$. 
%There exists an interemediate ring  $k \subset R \subset k'$ of finite type over $k$
%and a fibration of RCC fibrations $X\times_k R \supset Y^0_R \to W_R$ over $\Spec\,R$ 
%whose base to $k'$ is $X\times_k k' \supset Y^0 \to W$. 
%Thus, take a general closed point of $\Spec\,k_1 \in \Spec\,R$ and we obtain 
%an RCC fibration $X \times_k {k_1} \supset Y^0_1 \to W_1$ with $\dim W=\dim W_1$. 
%Thus, by \cite[(4.17.2) and Proof of Theorem~5.2]{Kollar1}, there exists a dominant rational map 
%$V \dashrightarrow W_1$. 
%Thus we see 
%$$\dim V \geq \dim W=\dim W_1.$$
%\end{proof}

\begin{prop}\label{MRCC-bc}
Let $k \subset k'$ be a field extension. 
Let $X$ be a proper geometrically normal and geometrically connected variety over $k$. 
Let $X \supset X^0 \to V$ and $X\otimes_k k' \supset Y^0 \to W$ 
be MRCC fibrations of $X$ and $X\otimes_k k'$, respectively. 
Then $\dim V=\dim W$. 
\end{prop}

\begin{proof}
For a proper variety $V$, 
set $r(V):=\dim V - \dim W$ where 
$V \supset V^0 \to W$ is an MRCC fibration. 
Note that $r(V)$ is well-defined 
because $W$ and $W_1$ are birational for another MRCC fibration $V \supset V^0_1 \to W_1$. 
It is enough to show $r(X)=r(X\otimes_k k')$. 

We show $r(X)\leq r(X\otimes_k k')$. 
Take an RCC fibration $X\supset X^0 \to Z$ such that 
the dimension of a general fiber is $r(X)$. 
Then, taking the base change to $k'$, 
we obtain an RCC fibration $X\otimes_k {k'} \supset X^0\otimes_k{k'} \to Z\otimes_k {k'}$. 
Therefore $r(X)\leq r(X\otimes_k k')$. 

\medskip

It suffices to prove $r(X)\geq r(X\times_k k')$. 
%Set $r(X\times_k k')$. 
We obtain an RCC fibration $X\otimes_k k'\supset Y^0 \to Z$ such that 
$\dim X-\dim Z=r(X\otimes_k k')$. 
%We find an RCC fibration $X \supset X^0 \to W$ 
%such that $\dim X-\dim W\geq r(X\times_k k')$. 

We prove that we may assume that $[k':k]<\infty$. 
We obtain a family of RCC fibrations $X\otimes_k R\supset Y^0_R \to Z_R$ over some intermediate ring 
$k\subset R \subset k'$ which is of finite type over $k$. 
Taking a general closed point of $\Spec\,R$, we may assume $[k':k]<\infty$. 

If $k'/k$ is purely inseparable, 
then the assertion can be easily proved. 
Thus, we may assume that $k'/k$ is a finite separable extension. 
Moreover, by taking the Galois closure of $k'/k$, 
we may assume that $k'/k$ is a finite Galois extension. 

\medskip 
Assume that $k'/k$ is a finite Galois extension. 
We have an MRCC fibration $Y:=X\otimes_k k'\supset Y^0 \to Z$ 
such that $\dim Y -\dim Z=r(X\otimes_k k')$. 
%It suffices to find an RCC fibration $X \supset X^0 \to W$ such that 
%$\dim X-\dim W\geq r(X\times_k k')$. 
%By replacing $Y\supset Y^0 \to Z$ with an MRCC fibration, 
%we may assume that $Y\supset Y^0 \to Z$ is an MRCC fibration. 
Since $k'/k$ is a Galois extension, so is $K(Y)=K(X\otimes_k k')/K(X)$. 
Set 
$$G:={\rm Gal}(K(Y)/K(X))=\{\sigma_1, \cdots, \sigma_N\}.$$
By shrinking $Z$, we may assume $Z=\Spec\,R_Z$. 
We fix an embedding $K(Z) \subset K(Y)$. 
Let $\sigma_i^*:Y \to Y$ be the induced automorphism.  

We show that $K(Z)=\sigma_i(K(Z))$ for every integer $1\leq i\leq N$. 
Fix $1\leq i\leq N$. 
We obtain another MRCC fibration 
$$Y=\sigma^*_i(Y)\supset \sigma^*_i(Y^0) \to \Spec\,\sigma_i(R_Z).$$
Then, we see that $\Spec\,R_Z$ and $\Spec\,\sigma_i(R_Z)$ are birational. 
This implies $K(Z)=\sigma_i(K(Z)).$ 

Thus, each $\sigma_i$ induces an birational automorphism $\sigma_i:Z \dashrightarrow Z$. 
Since $G=\{\sigma_1, \cdots, \sigma_N\}$ is a finite group, 
we can find a non-empty open subset $Z' \subset Z$ 
such that the induced rational map $Z' \dashrightarrow Z'$ is an isomorphism, i.e., automorphism. 
By replacing $Z$ with $Z'$, we may assume that $\sigma_i:Z \to Z$ is an automorphism. 
In particular $\sigma_i(R_Z)=R_Z$. 

We show $Y^0=\sigma_i^*(Y^0)$ for every $i$. 
By Lemma~\ref{rational-map}, we obtain 
$$Y^0=\{y\in Y\,|\, R_Z \subset \MO_{Y, y}\}.$$ 
Therefore, 
\begin{eqnarray*}
\sigma_i^*(Y^0)
&=&\sigma_i^*(\{y\in Y\,|\, R_Z \subset \MO_{Y, y}\})\\
&=&\{y':=\sigma_i^*(y)\in Y\,|\, R_Z \subset \MO_{Y, y}\}\\
&=&\{y'\in Y\,|\, R_Z \subset \MO_{Y, (\sigma_i^*)^{-1}(y')}\}\\
&=&\{y'\in Y\,|\, R_Z \subset \sigma_i^{-1}(\MO_{Y, y'})\}\\
&=&\{y'\in Y\,|\, \sigma_i(R_Z) \subset \MO_{Y, y'}\}\\
&=&\{y'\in Y\,|\, R_Z \subset \MO_{Y, y'}\}\\
&=&Y^0.
\end{eqnarray*}

Set $X^0:=\beta(Y^0)$, where $\beta:Y \to X$. 
Then, by $Y^0=\sigma_i(Y^0)$, we obtain 
$$\beta^{-1}(X^0)=\bigcup_{1\leq i\leq N} \sigma_i(Y^0)=\bigcup_{1\leq i\leq N} Y^0=Y^0.$$
By Lemma~\ref{quotient-fibration}, we obtain the following commutative diagram
$$\begin{CD}
Y^0@>>> X^0\\
@VVV @VVV\\
T @>>> S,
\end{CD}$$ 
where $S$ is an affine normal $k$-variety, 
$T\to S$ is a finite surjective morphism, 
and $X^0 \to S$ is a proper surjective morphism. 
Since fibers of $Y^0 \to T$ are RCC, so are the fibers of $X^0 \to S$. 
Therefore $X \supset X^0 \to S$ is an RCC fibration, which implies 
$$r(X)\geq \dim X -\dim S=\dim Y-\dim T=r(Y)=r(X\otimes_k k').$$
\end{proof}

\end{document}